\documentclass[11pt]{amsart}
\usepackage{mathrsfs}
\usepackage{amsfonts}
\usepackage{latexsym,amsmath,amssymb}

 \renewcommand{\epsilon}{\varepsilon}

\newtheorem{theorem}{Theorem}[section]

 \newtheorem{lemma}[theorem]{Lemma}
 
 \newtheorem{Corollary}[theorem]{Corollary}
 \newtheorem{proposition}[theorem]{proposition}
 \newtheorem{Proposition}[theorem]{Proposition}
\newtheorem{deff}[theorem]{Definition}
 \newtheorem{rem}[theorem]{Remark}
 \newcommand{\bth}{\begin{theorem}}
 \newcommand{\ble}{\begin{lemma}}
 \newcommand{\bcor}{\begin{corr}}
 \newcommand{\bdeff}{\begin{deff}}
 \newcommand{\bprop}{\begin{proposition}}
 \newcommand{\ele}{\end{lemma}}
 \newcommand{\ecor}{\end{corr}}
 \newcommand{\edeff}{\end{deff}}
 
 \newcommand{\eprop}{\end{proposition}}

 \renewcommand{\Pi}{\varPi}

 \renewcommand{\epsilon}{\varepsilon}

\numberwithin{equation}{section}

\thanks{The first author was supported  by National Science Foundation of China(No.11771103) and  Guangxi Natural Science Foundation(No.2016GXNSFCA380010).
The second author was supported by Guangdong Natural Science Foundation (No.2016A030307008).}

\title
[Nonlinear second boundary conditions ]{On the second boundary value problem for  a class of fully nonlinear flows I }
\author{Rongli Huang}
\address{School of Mathematics and Statistics, Guangxi Normal University,
Guilin, Guangxi 541004, People's Republic of China,
 E-mail: ronglihuangmath@gxnu.edu.cn}
 \author{Yunhua Ye}
\address{School of Mathematics, Jiaying University,
Meizhou, Guangdong 514015, People's Republic of China,
 E-mail: mathyhye@163.com}
\date{}

\begin{document}
\maketitle
\begin{abstract}

In this paper, a class of fully nonlinear flows   with nonlinear Neumann type boundary condition is considered.
This problem was solved partly by the first author under the assumption that the flow is the parabolic type special Lagrangian equation in $\mathbb{R}^{2n}$.
 We show that  the convexity is preserved for solutions of the fully nonlinear parabolic equations and  prove the long time existence and convergence of the flow.
In particular, we can prescribe the second boundary value problems for a family of special Lagrangian graphs in  Euclidean and  pseudo-Euclidean space.
\end{abstract}

\let\thefootnote\relax

\section{Introduction}
In this paper we will study the long time existence and convergence of convex solutions solving
\begin{equation}\label{e1.1}
\frac{\partial u}{\partial t}=F(D^{2}u),\quad
 \mathrm{in}\quad \Omega_{T}=\Omega\times(0,T),
\end{equation}
associated with the second boundary value problem
\begin{equation}\label{e1.2}
Du(\Omega)=\tilde{\Omega},  \quad t>0,
\end{equation}
and the initial condition
\begin{equation}\label{e1.3}
u=u_{0},  \quad t=0,
\end{equation}
for given $F$, $u_{0}$, $\Omega$  and  $\tilde{\Omega}$. Specifically, $\Omega$, $\tilde{\Omega}$  are uniformly convex bounded
domains with smooth boundary in $\mathbb{R}^{n}$.
$F$ is a $C^{2+\alpha_{0}}$ function for some $0<\alpha_{0}<1$ defined on the cone $\Gamma_{+}$ of positive definite symmetric matrices, which is monotonically increasing and
\begin{equation}\label{e1.4aa}
\left\{ \begin{aligned}&F[A]:=F(\lambda_{1},\lambda_{2},\cdots, \lambda_{n})\\
&F(\cdots,\lambda_{i},\cdots,\lambda_{j},\cdots )=F(\cdots,\lambda_{j},\cdots,\lambda_{i},\cdots),\quad \text{for} \,\,\, 1\leq i<j\leq n,
\end{aligned} \right.
\end{equation}
with
$$\lambda_{1}\leq\lambda_{2}\leq\cdots\leq\lambda_{n}$$
being the eigenvalues of the $n\times n$ symmetric matrix $A$.

Our motivation for studying equations (\ref{e1.1})-(\ref{e1.3}) comes from  providing a parabolic approach to prescribe the boundary value problems
for a family of special Lagrangian graphs. In Euclidean space $\mathbb{R}^{2n}$,  Brendle-Warren's theorem \cite{SM} says that
there exists a diffeomorphism f: $\Omega\rightarrow\tilde{\Omega}$ such that
 $$\Sigma=\{(x,f(x))|x\in \Omega\}$$
is a special Lagrangian submanifold. It's well known that $\Sigma$ is special Lagrangian
if and only if the Lagrangian angle is a constant,  a proof of which can be found in Harvey and Lawson's work (see Proposition 2.17 in \cite{HL}).
To find a special Lagrangian graph $\Sigma$ with $f$ being  a diffeomorphism : $\Omega\rightarrow\tilde{\Omega}$
is equivalent to the following problem ($f=Du$):
\begin{equation}\label{e1.4a}
\left\{ \begin{aligned}\Sigma_{i=1}^{n}\arctan\lambda_{i}&=c,
&  x\in \Omega, \\
Du(\Omega)&=\tilde{\Omega}.
\end{aligned} \right.
\end{equation}
Using the continuity method of solving fully nonlinear elliptic equations with second boundary condition,
S. Brendle and M. Warren \cite{SM} obtained the existence and uniqueness of the solution to (\ref{e1.4a}).
In \cite{HR}, the first author studied the parabolic type special Lagrangian equation with second boundary condition:
\begin{equation}\label{e1.5}
\left\{ \begin{aligned}\frac{\partial u}{\partial t}&=\Sigma_{i=1}^{n}\arctan\lambda_{i},
& t>0,\quad x\in \Omega, \\
Du(\Omega)&=\tilde{\Omega}, &t>0,\qquad\qquad\\
 u&=u_{0}, & t=0,\quad x\in \Omega.
\end{aligned} \right.
\end{equation}
The first author proved long time existence and convergence of the flow  and then  showed
the existence result of special Lagrangian submanifold with the same boundary condition in $\mathbb{R}^{n}\times\mathbb{R}^{n}$.
In this paper,  we will see that there exists relative interesting geometric flow  with second boundary condition,
where the stable solutions correspond to a family of special Lagrangian graphs in pseudo-Euclidean space. To obtain the long time existence and convergence of the flow,
it is often important to know whether the convexity of
solutions of the evolution equations involved remain unchanged in time, seeing the work of  B. Andrews \cite{B},  P. L.Lions and M. Musiela \cite{LM}.

Let $\delta_{0}$ be the standard Euclidean metric and $g_{0}$ be the pseudo-Euclidean metric where
$$g_{0}=\frac{1}{2}\sum_{i}(dx_{i}\bigotimes dy_{i}+dy_{i}\bigotimes dx_{i}).$$
By taking linear combinations of the metrics $\delta_{0}$ and $g_{0}$,  M. Warren constructed a family of metrics on  $\mathbb{R}^{n}\times\mathbb{R}^{n}$:
\begin{equation}
g_{\tau}=\cos \tau g_{0}+\sin \tau \delta_{0}.
\end{equation}
In calibrated geometry,  we adapt the details of the extremal Lagrangian surfaces in ($\mathbb{R}^{n}\times\mathbb{R}^{n}, g_{\tau})$ which were firstly obtained by
 M. Warren \cite{MW} and arise as
solutions to a family of special Lagrangian equations:
\begin{equation}\label{e1.6}
\,\,\,(\tau=0) \qquad \sum_{i}\ln\lambda_{i}=c,
\end{equation}
\begin{equation}\label{e1.7}
\,\,\,\, (0<\tau<\frac{\pi}{4}) \qquad \sum_{i}\ln(\frac{\lambda_{i}+a-b}{\lambda_{i}+a+b})=c,
\end{equation}
\begin{equation}\label{e1.8}
\,\,\,(\tau=\frac{\pi}{4}) \qquad \sum_{i}\frac{1}{1+\lambda_{i}}=c,
\end{equation}
\begin{equation}\label{e1.9}
\,\,\,(\frac{\pi}{4}<\tau<\frac{\pi}{2}) \qquad \sum_{i}\arctan(\frac{\lambda_{i}+a-b}{\lambda_{i}+a+b})=c,
\end{equation}
\begin{equation}\label{e1.10}
\,\,\,(\tau=\frac{\pi}{2}) \qquad \sum_{i}\arctan\lambda_{i}=c,
\end{equation}
where $a=\cot\tau$  and $b=\sqrt{|\cot^{2}\tau-1|}$.

Here we consider the following fully nonlinear elliptic equations with second boundary condition:
\begin{equation}\label{e1.12}
\left\{ \begin{aligned}F_{\tau}(D^{2}u)&=c,
&  x\in \Omega, \\
Du(\Omega)&=\tilde{\Omega},
\end{aligned} \right.
\end{equation}
where
$$F_{\tau}(A)=\left\{ \begin{aligned}
&\sum_{i}\ln\lambda_{i} , \quad
\qquad\quad\qquad\qquad\quad\quad\qquad\quad  \tau=0, \\
& \sum_{i}\ln(\frac{\lambda_{i}+a-b}{\lambda_{i}+a+b}) \qquad\qquad\quad\quad\qquad\quad 0<\tau<\frac{\pi}{4},\\
& \sum_{i}\frac{1}{1+\lambda_{i}}, \qquad \qquad\qquad \qquad\qquad \qquad\tau=\frac{\pi}{4},\\
& \sum_{i}\arctan(\frac{\lambda_{i}+a-b}{\lambda_{i}+a+b}), \qquad \qquad\qquad \quad\frac{\pi}{4}<\tau<\frac{\pi}{2},\\
& \sum_{i}\arctan\lambda_{i}, \,\,\quad \qquad\qquad \qquad\qquad \qquad\tau=\frac{\pi}{2},
\end{aligned} \right.$$
and $Du$ are the special Lagrangian  diffeomorphisms from $\Omega$ to $\tilde{\Omega}$ in  Euclidean and pseudo-Euclidean space.
In dimension 2, P. Delano\"{e} \cite{P} obtained a  unique smooth solution for the second boundary value problem of the Monge-Amp$\grave{e}$re equation
with respect to $\tau=0$ in (\ref{e1.12}) if  both domains are uniformly convex. Later the generalization of P. Delano\"{e}'s theorem to higher dimensions was given by
L. Caffarelli \cite{L} and J. Urbas \cite{JU}. Using the parabolic methods,  O.C. Schn$\ddot{\text{u}}$rer and K. Smoczyk \cite{OK} also obtained the existence of solutions to (\ref{e1.12}) for $\tau=0$.
As far as $\tau=\frac{\pi}{2}$ is concerned, S. Brendle and M. Warren \cite{SM} proved the existence and uniqueness of the solution  by the elliptic methods and the first author
\cite{HR} obtained the existence of solution  by the parabolic methods.

The aim of the paper is to study the existence of solutions of the equations (\ref{e1.12}) for $\tau=\frac{\pi}{4}$, i.e.,
\begin{equation}\label{e1.13}
\left\{ \begin{aligned}\sum_{i}\frac{1}{1+\lambda_{i}}&=c,
&  x\in \Omega, \\
Du(\Omega)&=\tilde{\Omega}.
\end{aligned} \right.
\end{equation}
 We will see that $-F_{\frac{\pi}{4}}$ can be viewed as the map $F$ defined by ($\ref{e1.4aa}$).

For any $\mu_{1}>0, \mu_{2}>0$, we define
$$\Gamma^{+}_{]\mu_{1},\mu_{2}[}=\{(\lambda_{1},\lambda_{2},\cdots, \lambda_{n})|0\leq\lambda_{1}\leq\lambda_{2}\leq\cdots\leq\lambda_{n}, \lambda_{1}\leq \mu_{1}, \lambda_{n}\geq \mu_{2}\}.$$
We assume that there exist  positive constants $\lambda, \Lambda$ depending only on $\mu_{1}, \mu_{2}$ such that
for any $(\lambda_{1},\lambda_{2},\cdots, \lambda_{n})\in \Gamma^{+}_{]\mu_{1},\mu_{2}[}$:
\begin{equation}\label{e1.16}
 \Lambda\geq\sum^{n}_{i=1}\frac{\partial F}{\partial \lambda_{i}}\geq \lambda,
\end{equation}
\begin{equation}\label{e1.17}
  \Lambda\geq\sum^{n}_{i=1}\frac{\partial F}{\partial \lambda_{i}}\lambda^{2}_{i}\geq \lambda.
\end{equation}
In addition,
\begin{equation}\label{e1.15}
 F(A)\,\, and\,\,F^{*}(A)\triangleq-F(A^{-1})\,\,are\,\, concave\,\, on\,\,\Gamma_{+}.
\end{equation}
Moreover, we assume that there exist two functions $f_{1}$, $f_{2}$ which are monotonically increasing in $(0,+\infty)$ satisfying
\begin{equation}\label{e1.15a}
 f_{1}(\lambda_{1})\leq F(\lambda_{1},\lambda_{2},\cdots, \lambda_{n})\leq f_{2}(\lambda_{n})\quad (\forall\,\,\, 0\leq\lambda_{1}\leq\lambda_{2}\leq\cdots\leq\lambda_{n}),
\end{equation}
and for any $\Phi, \Psi \in \pounds$,
\begin{equation}\label{e1.15b}
\left\{ \begin{aligned}
 f_{1}(t)\leq \Phi\Rightarrow \exists t_{1}>0,\,\, t\leq  t_{1} ,\\
 f_{2}(t)\geq \Psi\Rightarrow \exists t_{2}>0,\,\, t\geq t_{2} ,
 \end{aligned} \right.
\end{equation}
where
$$\pounds=\{\Upsilon| \exists(\lambda_{1},\lambda_{2},\cdots, \lambda_{n}),  0<\lambda_{1}\leq\lambda_{2}\leq\cdots\leq\lambda_{n}, \Upsilon=F(\lambda_{1},\lambda_{2},\cdots, \lambda_{n})\}.$$

We can't expect that $F$ satisfies (\ref{e1.16}) and (\ref{e1.17}) for the universal constants $\lambda$ and $\Lambda $ on  the cone $\Gamma_{+}$. The reason is in the following:
for any  $\epsilon>0$, by taking  $$\lambda_{1}=\lambda_{2}=\cdots=\lambda_{n}=\epsilon, $$
we obtain
$$\epsilon^{2}\Lambda\geq\epsilon^{2}\sum^{n}_{i=1}\frac{\partial F}{\partial \lambda_{i}}
=\sum^{n}_{i=1}\frac{\partial F}{\partial \lambda_{i}}\lambda^{2}_{i}\geq\lambda.$$
 In view of the above fact, we introduce the domain $\Gamma^{+}_{]\mu_{1},\mu_{2}[}$ such that the  two conditions are compatible. As same as the statement in \cite{SM}, the range of $c$  should be limited for the solvability of the equation  (\ref{e1.12}) and  the condition (\ref{e1.15a}) reflect the issue  in some
way. On the other hand, in Section 3,  we will prove that there exist universal constants  $\mu_{1}$ and $\mu_{2}$ such that
$(\lambda_{1},\lambda_{2},\cdots, \lambda_{n})$ are always in  $\Gamma^{+}_{]\mu_{1},\mu_{2}[}$ under the flow. So $F$ satisfies the structure
conditions (\ref{e1.16}) and (\ref{e1.17}) for the constants $\lambda$ and $\Lambda$ under the flow.

By \cite{HR} and Corollary \ref{c5.3} of the present paper,  the  examples of functions satisfying  (\ref{e1.16})- (\ref{e1.15b}) are those corresponding to
\begin{equation*}
  \left\{ \begin{aligned}
 &F(\lambda_{1},\lambda_{2},\cdots, \lambda_{n}):=F_{\frac{\pi}{2}}(\lambda_{1},\lambda_{2},\cdots, \lambda_{n})=\sum_{i}\arctan\lambda_{i},\\
 &f_{1}(t)=n\arctan t,\,\,\,f_{2}(t)=n\arctan t ,
 \end{aligned} \right.
\end{equation*}
and
\begin{equation*}
  \left\{ \begin{aligned}
 &F(\lambda_{1},\lambda_{2},\cdots, \lambda_{n}):=-F_{\frac{\pi}{4}}(\lambda_{1},\lambda_{2},\cdots, \lambda_{n})=-\sum_{i}\frac{1}{1+\lambda_{i}},\\
 &f_{1}(t)=-\frac{n}{1+t},\,\,\,f_{2}(t)=-\frac{n}{1+t} .
 \end{aligned} \right.
\end{equation*}

The main results of this paper can be summarized as follows
\begin{theorem}\label{t1.1}
Assume that $\Omega$, $\tilde{\Omega}$ are bounded, uniformly convex domains with smooth boundary in $\mathbb{R}^{n}$, $0<\alpha_{0}<1$ and
the map $F$ satisfies (\ref{e1.4aa}),  (\ref{e1.16}), (\ref{e1.17}), (\ref{e1.15}), (\ref{e1.15a}), (\ref{e1.15b}).
  Then for any given initial function $u_{0}\in C^{2+\alpha_{0}}(\bar{\Omega})$
  which is   uniformly convex and satisfies $Du_{0}(\Omega)=\tilde{\Omega}$,  the  strictly convex solution of (\ref{e1.1})-(\ref{e1.3}) exists
  for all $t\geq 0$ and $u(\cdot,t)$ converges to a function $u^{\infty}(x,t)=u^\infty(x)+C_{\infty}\cdot t$ in $C^{1+\zeta}(\bar{\Omega})\cap C^{4+\alpha}(\bar{D})$ as $t\rightarrow\infty$
  for any $D\subset\subset\Omega$, $\zeta<1$,$0<\alpha<\alpha_{0}$, i.e.,

  $$\lim_{t\rightarrow+\infty}\|u(\cdot,t)-u^{\infty}(\cdot,t)\|_{C^{1+\zeta}(\bar{\Omega})}=0,\qquad
  \lim_{t\rightarrow+\infty}\|u(\cdot,t)-u^{\infty}(\cdot,t)\|_{C^{4+\alpha}(\bar{D})}=0.$$
And $u^{\infty}(x)\in C^{1+1}(\bar{\Omega})\cap C^{4+\alpha}(\Omega)$ is a solution of
\begin{equation}\label{e1.18}
\left\{ \begin{aligned}F(D^{2}u)&=C_{\infty},
&  x\in \Omega, \\
Du(\Omega)&=\tilde{\Omega}.
\end{aligned} \right.
\end{equation}
The constant $C_{\infty}$ depends only on $\Omega$, $\tilde{\Omega}$ and $F$. The solution to (\ref{e1.18}) is unique up to additions of constants.
\end{theorem}
\begin{Corollary}\label{r1.2}
Let  $\Omega$, $\tilde{\Omega}$ and $F$ satisfy the conditions in the above theorem. If $F$ is $C^{\infty}$, then there exist $u^{\infty}(x)\in C^{\infty}(\bar{\Omega})$ and the constant $C_{\infty}$ which satisfy (\ref{e1.18}).
\end{Corollary}
\begin{rem}\label{r1.3}
Theorem 3.3 of B.Andrews in \cite{B} proved that the convexity for solutions of fully nonlinear parabolic equations under conditions
(\ref{e1.4aa}) and  (\ref{e1.15}) can be preserved if the solutions on the boundary are strictly convex. Here we do not make any assumption of $u$ to be convex on the boundary.
\end{rem}
\begin{rem}\label{r1.4}
The first author \cite{HRO} used the  elliptic methods to obtain the similar results of (\ref{e1.18}),
but the structure conditions of F in \cite{HRO} are more complicated.
\end{rem}
As a consequence of Theorem \ref{t1.1}, we will prove the existence and uniqueness of classical solution of (\ref{e1.13}).
\begin{theorem}\label{t1.2}
Let $\Omega$, $\tilde{\Omega}$ be bounded, uniformly convex domains with smooth boundary in $\mathbb{R}^{n}$.
  Then  there exist  $u\in  C^{\infty}(\bar{\Omega})$  and the constant $c$ such that
  $u$ is a solution of  the second boundary value problem (\ref{e1.13}). The solution to (\ref{e1.13}) is unique up to additions of constants.
\end{theorem}

\begin{deff}\label{d1.1}
 We say that $\Sigma=\{(x,f(x))|x\in \Omega\}$ is a special Lagrangian graph in $(\mathbb{R}^{n}\times\mathbb{R}^{n}, g_{\tau})$ if $$f=Du$$
  and $u$ satisfies
 $$F_{\tau}(D^{2}u (x))=c,\quad x\in \Omega.$$
\end{deff}

By the above result, we can extend Brendle-Warren's  theorem \cite{SM} to the case in ($\mathbb{R}^{n}\times\mathbb{R}^{n}, g_{\frac{\pi}{4}}$):
\begin{Corollary}\label{c1.4}
Let $\Omega$, $\tilde{\Omega}$ be bounded, uniformly convex domains with smooth boundary in $\mathbb{R}^{n}$. Then
there exists a diffeomorphism f: $\Omega\rightarrow\tilde{\Omega}$ such that
 $$\Sigma=\{(x,f(x))|x\in \Omega\}$$
is a special Lagrangian graph in ($\mathbb{R}^{n}\times\mathbb{R}^{n}, g_{\frac{\pi}{4}}$).
\end{Corollary}

The monotone increasing of $F$  implies the ellipticity of the equation (\ref{e1.18}). As  the statement in \cite{SM}, the range of $c$ should be limited for the solvability of the equation  (\ref{e1.18}), so the conditions (\ref{e1.15a})-(\ref{e1.15b}) are just  natural  in some degree.    Condition (\ref{e1.15}) is  essentially
the ones used in \cite{B}  to preserve the convexity for the solutions of the fully nonlinear parabolic equations. And conditions (\ref{e1.16}), (\ref{e1.17})
are essential part to carry out the upcoming $C^{2}$ priori estimates for the solutions on the boundary.

The rest of this paper is organized as follows. In Section 2, we establish the short time existence result to the flow (\ref{e1.1})-(\ref{e1.3}) by
using the inverse function theorem. In Section 3, we collect necessary preliminaries which will be used in the proof of Theorem \ref{t1.1}.
The techniques used in this section are reflective of those in \cite{HR}, \cite{JU},  and \cite{OK} to the second boundary value problems for fully nonlinear differential equations.
We use  barrier arguments to obtain the $C^{2}$ upper bound estimates on the boundary. Here the structure conditions (\ref{e1.16}), (\ref{e1.17})
play an important role to construct suitable auxiliary functions as barriers.  The conditions  are also used to establish the $C^{2}$ lower bounds for the solutions on the boundary.
   In Section 4, we give the proof of Theorem \ref{t1.1} and Corollary \ref{r1.2}. In Section 5, we replace $F$ by $-F_{\frac{\pi}{4}}$
 in (\ref{e1.1}) and solve the corresponding problems. Under the flow (\ref{e1.1})-(\ref{e1.3}), we  show that $-F_{\frac{\pi}{4}}$ satisfies the conditions  (\ref{e1.16})-(\ref{e1.15b}),
 and then we  give the proof of  Theorem \ref{t1.2} and Corollary \ref{c1.4}. In Section 6, we show that the applications of  Theorem \ref{t1.1}  include a number of previously
established results of various authors as consequences.

\section{ The short time existence of the
parabolic flow }

 Throughout the following,  Einstein's convention of
 summation over repeated indices will be adopted.
Denote $$u_{i}=\dfrac{\partial u}{\partial x_{i}},
u_{ij}=\dfrac{\partial^{2}u}{\partial x_{i}\partial x_{j}},
u_{ijk}=\dfrac{\partial^{3}u}{\partial x_{i}\partial x_{j}\partial
x_{k}}, \cdots $$ and
$$[u^{ij}]=[ u_{ij}]^{-1},\,\,\,
F^{ij}(D^{2}u)=\frac{\partial F}{\partial u_{ij}},\,\,\, \Omega_{T}=\Omega\times(0,T). $$

We first recall the relevant Sobolev spaces ( cf. Chapter 1 in  \cite{W}). For every multi-index $\beta=(\beta_{1},\beta_{2},\cdots, \beta_{n})(\beta_{i}\geq 0$
for $i=1,2,\cdots,n)$ with length $|\beta|=\sum^{n}_{i=1}\beta_{i}$ and $j\geq 0$, we set
$$D^{\beta}u=\frac{\partial^{|\beta|}u}{\partial x_{1}^{\beta_{1}}\partial x_{2}^{\beta_{2}}\cdots\partial x_{n}^{\beta_{n}}},\qquad
D^{\beta}D_{t}^{j}u=\frac{\partial^{|\beta|+j}u}{\partial x_{1}^{\beta_{1}}\partial x_{2}^{\beta_{2}}\cdots\partial x_{n}^{\beta_{n}}\partial t^{j}}.$$
We remind the definition of the usual functional spaces ($k\geq 0$):\\
\\$C^{k}(\Omega)=\{u:\Omega\rightarrow \mathbb{R}|\forall \beta$, $|\beta|\leq k, D^{\beta}u$ is continuous in $\Omega\}$,\\
\\$C^{k}(\bar{\Omega})=\{u\in C^{k}(\Omega) |\forall \beta$, $|\beta|\leq k, D^{\beta}u$ can be extended by continuity to $\partial\Omega\}$;\\
\\$C^{k,\frac{k}{2}}(\Omega_{T})=\{u:\Omega_{T}\rightarrow \mathbb{R}|\forall \beta,j\geq 0$, $|\beta|+2j\leq k, D^{\beta}D^{j}_{t}u$ is continuous in $\Omega_{T}\}$,\\
\\$C^{k,\frac{k}{2}}(\bar{\Omega}_{T})=\{u\in C^{k,\frac{k}{2}}(\Omega_{T}) |\forall \beta,j\geq 0$, $|\beta|+2j\leq k, D^{\beta}D^{j}_{t}u$ can be extended by continuity to $\partial\Omega_{T}\}$.\\
\\
Moreover $C^{k}(\bar{\Omega}), C^{k,\frac{k}{2}}(\bar{\Omega}_{T})$ are Banach spaces equipped with the norm respectively:
$$\|u\|_{C^{k}(\bar{\Omega})}=\sum^{k}_{i=0}\sup_{|\beta|=i}\sup_{\bar{\Omega}}|D^{\beta}u|,$$
$$\|u\|_{C^{k,\frac{k}{2}}(\bar{\Omega}_{T})}=\sum_{j\geq 0,|\beta|+2j\leq k}\sup_{\bar{\Omega}_{T}}|D^{\beta}D^{j}_{t}u|.$$
We now remind the definition of H$\ddot{\text{o}}$lder spaces. Let $\alpha\in [0,1]$. We define the $\alpha$-H$\ddot{\text{o}}$lder coefficient of $u$ in $\Omega$:
$$[u]_{\alpha, \Omega}=\sup_{x\neq y, x,y\in\Omega}\frac{|u(x)-u(y)|}{|x-y|^{\alpha}}.$$
If $[u]_{\alpha, \Omega}<+\infty,$ then we call $u$ H$\ddot{\text{o}}$lder continuous with exponent $\alpha$ in $\Omega.$
If there are no ambiguity about the domains $\Omega$, we denote $[u]_{\alpha, \Omega}$ by
$[u]_{\alpha}$.
Similarly, the $(\alpha,\frac{\alpha}{2})$-H$\ddot{\text{o}}$lder coefficient of $u$ in $\Omega_{T}$ can be defined by
$$[u]_{\alpha,\frac{\alpha}{2}, \Omega_{T}}=\sup_{(x,t)\neq (y,\tau), (x,t),(y,\tau)\in\Omega_{T}}\frac{|u(x,t)-u(y,\tau)|}{\max\{|x-y|^{\alpha},|t-\tau|^{\frac{\alpha}{2}}\}},$$
and $u$ is H$\ddot{\text{o}}$lder continuous with exponent $(\alpha,\frac{\alpha}{2})$ in $\Omega_{T}$  if $[u]_{\alpha,\frac{\alpha}{2}, \Omega_{T}} <+\infty.$
Meanwhile, we denote $[u]_{\alpha,\frac{\alpha}{2}, \Omega_{T}}$ by  $[u]_{\alpha,\frac{\alpha}{2}}$.   We denote $C^{k+\alpha}(\bar{\Omega})$ as the set of functions belonging to
$C^{k}(\bar{\Omega})$ whose $k$-order partial derivatives are H$\ddot{\text{o}}$lder continuous with exponent $\alpha$ in $\Omega$ and $C^{k+\alpha}(\bar{\Omega})$ is
a Banach space equipped with the following norm:
$$\|u\|_{C^{k+\alpha}(\bar{\Omega})}=\|u\|_{C^{k}(\bar{\Omega})}+[u]_{k+\alpha}$$
where
$$[u]_{k+\alpha}=\sum_{|\beta|=k}[D^{\beta}u]_{\alpha}.$$
Likewise, we denote $C^{k+\alpha,\frac{k+\alpha}{2}}(\bar{\Omega}_{T})$ as the set of functions belonging to
$C^{k,\frac{k}{2}}(\bar{\Omega}_{T})$ whose $(k,\frac{k}{2})$-order partial derivatives are H$\ddot{\text{o}}$lder continuous with exponent $(\alpha,\frac{\alpha}{2})$ in $\Omega_{T}$ and $C^{k+\alpha,\frac{k+\alpha}{2}}(\bar{\Omega}_{T})$ is
a Banach space equipped with the following norm:
$$\|u\|_{C^{k+\alpha,\frac{k+\alpha}{2}}(\bar{\Omega}_{T})}=\|u\|_{C^{k,\frac{k}{2}}(\bar{\Omega}_{T})}+[u]_{k+\alpha, \frac{k+\alpha}{2}} , $$
where
$$[u]_{k+\alpha,\frac{k+\alpha}{2}}=\sum_{|\beta|+2j=k}[D^{\beta}D^{j}_{t}u]_{\alpha,\frac{\alpha}{2}}.$$

By the methods on the second boundary value problems for equations of Monge-Amp\`{e}re type \cite{JU},
the parabolic boundary condition in (\ref{e1.2}) can be reformulated as
$$h(Du)=0,\qquad x\in \partial\Omega,\quad t>0,$$
where $h$ is a smooth function on $\bar{\tilde{\Omega}}$:
$$\tilde{\Omega}=\{p\in\mathbb{R}^{n} |h(p)>0\},\qquad |Dh|_{{\partial\tilde{\Omega}}}=1.$$
The so called boundary defining function is strictly concave, i.e., $\exists \theta>0$,
$$ \frac{\partial^{2}h}{\partial y_{i}\partial y_{j}}\xi_{i}\xi_{j}\leq
-\theta|\xi|^{2},\qquad \text{for}\quad \forall p=(y_{1}, y_{2},\cdots, y_{n})\in \tilde{\Omega},\quad
 \xi=(\xi_{1}, \xi_{2},\cdots, \xi_{n})\in \mathbb{R}^{n}.$$
We also give the boundary defining function according to $\Omega$ (cf.\cite{SM}):
$$\Omega=\{p\in\mathbb{R}^{n} |\tilde{h}(p)>0\},\qquad |D\tilde{h}|_{{\partial\Omega}}=1,$$
$$\exists\tilde{\theta}>0,\quad\frac{\partial^{2}\tilde{h}}{\partial y_{i}\partial y_{j}}\xi_{i}\xi_{j}\leq
-\tilde{\theta}|\xi|^{2},\quad \text{for}\quad \forall p=(y_{1}, y_{2},\cdots, y_{n})\in \Omega,\quad
 \xi=(\xi_{1}, \xi_{2},\cdots, \xi_{n})\in \mathbb{R}^{n}.$$
Thus the parabolic flow (\ref{e1.1})-(\ref{e1.3}) is equivalent to the evolution problem:
\begin{equation}\label{e2.1}
\left\{ \begin{aligned}\frac{\partial u}{\partial t}&=F(D^{2}u),
& t>0,\quad x\in \Omega, \\
h(Du)&=0,& \qquad t>0,\quad x\in\partial\Omega,\\
 u&=u_{0}, & \qquad\quad t=0,\quad x\in \Omega.
\end{aligned} \right.
\end{equation}
To establish the short time existence of classical solutions of (\ref{e2.1}),  we use the inverse function theorem
in Fr$\acute{e}$chet spaces and the theory of linear parabolic equations for oblique boundary condition. The method is along the idea of
proving  the short time existence of convex solutions on the second boundary value problem for Lagrangian mean curvature flow \cite{HR}.   We include the details for the convenience
of the readers.
\begin{lemma}(I. Ekeland, see Theorem 2 in  \cite{IE}.)\label{l1.10}
Let $X$ and $Y$ be Banach spaces. Suppose $$\hbar: X\rightarrow Y$$
 is continuous and G$\hat{a}$teaux-differentiable, with
$\hbar[0]=0$. Assume that the derivative $D\hbar[x]$ has a right inverse $\mathrm{T}[x]$, uniformly bounded in
a neighbourhood of $0$ in $X$:
$$\forall y\in Y,\quad D\hbar[x]\mathrm{T}[x]y=y ;$$
$$\|x\|\leq R\Longrightarrow \|\mathrm{T}[x]\|\leq m.$$
For every $y\in Y$  if
$$\parallel y\parallel<\frac{R}{m},$$
then there is some $x\in X$ such that :
$$\|x\|< R$$
and
$$\hbar[x]=y.$$
\end{lemma}

As an application of I. Ekeland's theorem, we obtain the following inverse function theorem which will be used to prove the short time existence result for equation (\ref{e2.1}).
\begin{lemma}\label{l1.11}
Let $X$ and $Y$ be Banach spaces. Suppose $$J: X\rightarrow Y$$
 is continuous and G$\hat{a}$teaux-differentiable, with
$J(v_{0})=w_{0}$. Assume that the derivative $DJ[v]$ has a right inverse $L[v]$, uniformly bounded in
a neighbourhood of $v_{0}$:
$$\forall y\in Y,\quad DJ[v]L[v]y=y ;$$
$$\|v-v_{0}\|\leq R\Longrightarrow \|L[v]\|\leq m.$$
For every $w\in Y$  if
$$\parallel w-w_{0}\parallel<\frac{R}{m},$$
then there is some $v\in X$ such that:
$$\|v-v_{0}\|< R$$
and
$$J(v)=w.$$
\end{lemma}
\begin{proof}
Denote $v=x+v_{0}$ and $\hbar[x]\triangleq J[x+v_{0}]-w_{0}$, then $\hbar[0]=0$. Since the derivative $DJ[v]$ has a right inverse $L[v]$,
we deduce that $D\hbar[x]=DJ[v]$ has a right inverse $L[v]$. Set $ \mathrm{T}[x]=L[x+v_{0}]$. Following Lemma \ref{l1.10},
for every $y\triangleq w-w_{0}\in Y$,  if
$$\parallel y\parallel<\frac{R}{m},$$
then there is some $x\in X$ such that:
$$\|x\|< R$$
and
$$\hbar[x]=y=w-w_{0}.$$
So  we have:
$$\|v-v_{0}\|< R,$$
and
$$J(v)=w.$$
\end{proof}

\begin{lemma}(See Theorem 8.8 and 8.9 in \cite{GM}.)\label{l1.2}
Assume that $f\in C^{\alpha_{0},\frac{\alpha_{0}}{2}}(\bar{\Omega}_{T})$  for some $0<\alpha_{0}<1$, $T>0$, and
$G(x,p)$, $G_{p}(x,p)$ are in $C^{1+\alpha_{0}}(\Xi)$ for any compact subset $\Xi$ of $\partial\Omega\times\mathbb{R}^{n}$
such that $\inf_{\partial\Omega}\langle G_{p}, \nu\rangle>0$
where $\nu$ is the inner normal vector of $\partial\Omega$. Let $u_{0}\in C^{2+\alpha_{0}}(\bar{\Omega})$ be strictly convex and satisfy
$G(x, Du_{0})=0.$ Then there exists  $T'>0$ $(T'\leq T)$ such that we can find a unique solution which is strictly convex in $x$ variable in the class $C^{2+\alpha_{0},\frac{2+\alpha_{0}}{2}}(\bar{\Omega}_{T'})$
to the following equations
\begin{equation*}\label{e1.4}
\left\{ \begin{aligned}\frac{\partial u}{\partial t}-a^{ij}(x,t) u_{ij}&=f(x,t),
& T'>t>0,\quad x\in \Omega, \\
G(x,Du)&=0,& \qquad T'>t>0,\quad x\in\partial\Omega,\\
 u&=u_{0}, & \qquad\quad t=0,\quad x\in \Omega,
\end{aligned} \right.
\end{equation*}
where  $a^{ij}(x,t)(1\leq i,j\leq n)\in C^{\alpha_{0},\frac{\alpha_{0}}{2}}(\bar{\Omega}_{T})$ and $[a^{ij}(x,t)]\geq a_{0}\text{I}$ for some positive constant $a_{0}$.
\end{lemma}
By the property of $C^{2+\alpha_{0},\frac{2+\alpha_{0}}{2}}(\bar{\Omega}_{T'})$ and $u(x,t)|_{t=0}=u_{0}(x)$, we obtain
\begin{equation}\label{e2.00}
\lim_{t\rightarrow 0}\parallel u(\cdot,t)-u_{0}(\cdot)\parallel_{C^{2+\alpha_{0}}(\bar{\Omega})}=0.
\end{equation}
For any $\alpha<\alpha_{0}$, we have
\begin{equation*}
\begin{aligned}
&\frac{|(D^{2}u(x,t)-D^{2}u_{0}(x))-(D^{2}u(y,\tau)-D^{2}u_{0}(y))|}{\max\{|x-y|^{\alpha},|t-\tau|^{\frac{\alpha}{2}}\}} \\
&\leq \frac{|(D^{2}u(x,t)-D^{2}u_{0}(x))-(D^{2}u(y,t)-D^{2}u_{0}(y))|}{|x-y|^{\alpha}}\\
&+|t-\tau|^{\frac{\alpha_{0}-\alpha}{2}}\frac{|(D^{2}u(y,t)-D^{2}u_{0}(y))-(D^{2}u(y,\tau)-D^{2}u_{0}(y))|}{|t-\tau|^{\frac{\alpha_{0}}{2}}}.
\end{aligned}
\end{equation*}
Then we get
\begin{equation}\label{e2.01}
\begin{aligned}
\parallel D^{2}u-D^{2}u_{0}\parallel_{C^{\alpha,\frac{\alpha}{2}}(\bar{\Omega}_{T'})}\leq &\max_{0\leq t\leq T'}\parallel D^{2}u(\cdot,t)-D^{2}u_{0}(\cdot)\parallel_{C^{\alpha}(\bar{\Omega})}\\
&+T'^{\frac{\alpha-\alpha_{0}}{2}}\parallel D^{2}u-D^{2}u_{0}\parallel_{C^{\alpha_{0},\frac{\alpha_{0}}{2}}(\bar{\Omega}_{T'})}.
\end{aligned}
\end{equation}
Combining  (\ref{e2.00}) with (\ref{e2.01}), we obtain
\begin{equation}\label{e2.02}
\lim_{T'\rightarrow 0}\parallel D^{2}u-D^{2}u_{0}\parallel_{C^{\alpha,\frac{\alpha}{2}}(\bar{\Omega}_{T'})}=0
\end{equation}
which will be used later.

According to the proof in \cite{JU}, we can verify the oblique boundary condition.
\begin{lemma}(See J. Urbas \cite{JU}.)\label{l1.3}\quad
\\
$u\in C^{2}(\bar{\Omega})$  with $D^{2}u>0$
$\Longrightarrow$ $\inf_{\partial\Omega}h_{p_{k}}(Du)\nu_{k}>0$ where $\nu=(\nu_{1},\nu_{2}, \cdots,\nu_{n})$ is the unit inward normal vector of $\partial\Omega$,
i.e., $h(Du)=0$ is strictly oblique.
\end{lemma}
We are now in a position to prove the short time existence of solutions of (\ref{e2.1})
which is equivalent to the problem (\ref{e1.1})-(\ref{e1.3}).
\begin{Proposition}\label{p1.1}
According to the conditions in Theorem \ref{t1.1}, there exist some $T''>0$ and $u\in C^{2+\alpha,\frac{2+\alpha}{2}}(\bar{\Omega}_{T''})$ which depend only on $\Omega$, $\tilde{\Omega}$, $u_{0}$,
 such that $u$ is  a solution of (\ref{e2.1}) and  is strictly convex in $x$ variable.
\end{Proposition}

\begin{proof}
Denote the Banach spaces
$$X=C^{2+\alpha,1+\frac{\alpha}{2}}(\bar{\Omega}_{T}),
\quad Y=C^{\alpha,\frac{\alpha}{2}}(\bar{\Omega}_{T})
\times C^{1+\alpha,\frac{1+\alpha}{2}}(\partial\Omega\times(0,T))\times C^{2+\alpha}(\bar{\Omega}),$$
where $$\parallel\cdot\|_{Y}=\parallel\cdot\|_{C^{\alpha,\frac{\alpha}{2}}(\bar{\Omega}_{T})}
+\parallel\cdot\|_{C^{1+\alpha,\frac{1+\alpha}{2}}(\partial\Omega\times(0,T))}+
\parallel\cdot\|_{C^{2+\alpha}(\bar{\Omega})}.$$
Define a map $$J:\quad X\rightarrow Y$$ by
$$J(u)=\left\{ \begin{aligned}
&\frac{\partial u}{\partial t}-F(D^{2}u), &\quad (x,t)\in\Omega_{T}, \\
&h(Du), &\quad (x,t)\in \partial\Omega\times(0,T),\\
&u, &\quad (x,t)\in \Omega\times\{t=0\}.
\end{aligned} \right.$$
The strategy is now to use the inverse function theorem to obtain the short time existence result.
The computation of the G$\hat{a}$teaux derivative shows that:
$$\forall u,v\in X,\quad DJ[u](v)\triangleq\frac{d}{d\tau}J(u+\tau v)|_{\tau=0}=\left\{ \begin{aligned}
&\frac{\partial v}{\partial t}-F^{ij}(D^{2}u)v_{ij}, &\quad (x,t)\in\Omega_{T}, \\
&h_{p_{i}}(Du)v_{i}, &\quad (x,t)\in \partial\Omega\times(0,T),\\
&v, &\quad (x,t)\in \Omega\times\{t=0\}.
\end{aligned} \right.$$
Using Lemma \ref{l1.2} and Lemma \ref{l1.3}, there exists $T_{1}>0$ such that  we can find $$\hat{u}\in C^{2+\alpha_{0},1+\frac{\alpha_{0}}{2}}(\bar{\Omega}_{T_{1}})\subset X$$ to be strictly convex in $x$ variable,  which satisfies the following equations :
\begin{equation}\label{e2.2}
\left\{ \begin{aligned}\frac{\partial \hat{u}}{\partial t}-\triangle \hat{u}&=F(D^{2}u_{0})-\triangle u_{0},
& T_{1}>t>0,\quad x\in \Omega, \\
h(D\hat{u})&=0,& \qquad T_{1}>t>0,\quad x\in\partial\Omega.\\
 \hat{u}&=u_{0}, & \qquad\quad t=0,\quad x\in \Omega.
\end{aligned} \right.
\end{equation}
We see that $\exists R>0$, such that $u$ is strictly convex in $x$ variable if $$\|u-\hat{u}\|_{C^{2+\alpha,\frac{2+\alpha}{2}}(\bar{\Omega}_{T_{1}})}<R,$$
 For each $Z\triangleq(f,g,w)\in Y $ and  using Lemma \ref{l1.2} again, we know that  there exists a unique $v\in X (T=T_{1})$ satisfying $DJ[u](v)=(f,g,w)$, i.e.
\begin{equation*}
\left\{ \begin{aligned}\frac{\partial v}{\partial t}-F^{ij}(D^{2}u)v_{ij}&=f,
& T_{1}>t>0,\quad x\in \Omega, \\
h_{p_{i}}(Du)v_{i}&=g,& \qquad T_{1}>t>0,\quad x\in\partial\Omega,\\
 v&=w, & \qquad\quad t=0,\quad x\in \Omega.
\end{aligned} \right.
\end{equation*}
Using Schauder estimates for linear parabolic equation to oblique boundary condition
(cf. Theorem 8.8 and 8.9 in \cite{GM}), we obtain
\begin{equation*}
\parallel v \parallel_{C^{2+\alpha,\frac{2+\alpha}{2}}(\bar{\Omega}_{T_1})}
\leq
m (\parallel f \|_{C^{\alpha,\frac{\alpha}{2}}(\bar{\Omega}_{T_1})}
+\parallel g \|_{C^{1+\alpha,\frac{1+\alpha}{2}}(\partial\Omega\times(0,T_1))}+
\parallel w \|_{C^{2+\alpha}(\bar{\Omega})}),
\end{equation*}
for some positive constant $m$.
Using the definition of the Banach spaces $X$ and $Y$ with $T=T_{1}$, we can rewrite the above Schauder estimates as
\begin{equation*}
\parallel v \parallel_{X} \leq m \parallel Z \parallel_{Y}.
\end{equation*}
If  $\|Z\|_{Y}\leq 1$, then we have $$\|v\|_{X}\leq m. $$
It means that the derivative $DJ[u](v)=Z$ has a right inverse $v=L[u](Z)$ and
\begin{equation*}\label{e2.3}
\|L[u]\|\triangleq \sup_{\|Z\|_{Y}\leq 1}\|L[u](Z)\|_{X}\leq m.
\end{equation*}
If  we set $$\hat{f}=\frac{\partial \hat{u}}{\partial t}-F(D^{2}\hat{u}),\quad w_{0}=(\hat{f}, 0,u_{0}),\quad
w=(0, 0,u_{0}),$$
then we can show that
\begin{equation*}
\begin{aligned}
\parallel\hat{f}-0\parallel_{C^{\alpha,\frac{\alpha}{2}}(\bar{\Omega}_{T_{1}})}
&=\parallel\triangle \hat{u}-\triangle u_{0}+F(D^{2}u_{0})-F(D^{2}\hat{u})\parallel_{C^{\alpha,\frac{\alpha}{2}}(\bar{\Omega}_{T_{1}})}\\
&\leq\parallel\triangle \hat{u}-\triangle u_{0}\parallel_{C^{\alpha,\frac{\alpha}{2}}(\bar{\Omega}_{T_{1}})}+\parallel F(D^{2}u_{0})-F(D^{2}\hat{u})\parallel_{C^{\alpha,\frac{\alpha}{2}}(\bar{\Omega}_{T_{1}})}\\
&\leq C\parallel D^{2}\hat{u}-D^{2}u_{0}\parallel_{C^{\alpha,\frac{\alpha}{2}}(\bar{\Omega}_{T_{1}})},
\end{aligned}
\end{equation*}
where $C$ is a constant depending only on the known data.  Using (\ref{e2.02}), we conclude:
 $\exists T''>0$  $(T''\leq T_{1})$ to be small enough such that
$$\parallel\hat{f}-0\parallel_{C^{\alpha,\frac{\alpha}{2}}(\bar{\Omega}_{T''})}\leq  C\parallel D^{2}\hat{u}-D^{2}u_{0}\parallel_{C^{\alpha,\frac{\alpha}{2}}(\bar{\Omega}_{T''})}<\frac{R}{m}.$$
Thus we obtain
$$\parallel w-w_{0}\|_{Y}=
\parallel0-\hat{f}\|_{C^{\alpha,\frac{\alpha}{2}}(\bar{\Omega}_{T''})}<\frac{R}{m}.
$$
By Lemma \ref{l1.11},  we give the desired results.
\end{proof}
\begin{rem}
By the strong maximum principle, the strictly convex solution to (\ref{e2.1}) is unique.
\end{rem}

\section{ Preliminary results }
In this section, the $C^{2}$ a priori bound is accomplished by making the second derivative estimates on the boundary
for the solutions of fully nonlinear parabolic equations. This treatment
is similar to the problems presented in \cite{HR}, \cite{JU} and  \cite{OK}, but requires some modification to
accommodate the more general   situation. Specifically, the structure conditions  (\ref{e1.16}) and (\ref{e1.17}) are needed in order to
derive differential inequalities from  barriers  which can be used.

For convenience of the computation, we set
$$
\beta^{k}\triangleq\frac{\partial h(Du)}{\partial u_{k}}=h_{p_{k}}(Du)$$
and  $\langle\cdot,\cdot\rangle$ be the inner product in $\mathbb{R}^{n}$.
By Proposition \ref{p1.1} and the regularity theory of parabolic equations,   we may assume that $u$ is a strictly convex solution of (\ref{e1.1})-(\ref{e1.3}) in the class $C^{2+\alpha,1+\frac{\alpha}{2}}(\bar{\Omega}_{T})\cap C^{4+\alpha,2+\frac{\alpha}{2}}(\Omega_{T})$ for some $T>0$.
\begin{lemma}[$\dot{u}$-estimates]\label{l5.1a}\quad
\\
As long as the convex solution to (\ref{e1.1})-(\ref{e1.3}) exists, the following estimates hold, i.e.,
\begin{equation*}\label{e3.1}
\Theta_{0}\triangleq\min_{\bar{\Omega}}F(D^{2}u_{0})\leq\dot{u}\triangleq\frac{\partial u}{\partial t}\leq\Theta_{1}\triangleq\max_{\bar{\Omega}}F(D^{2}u_{0}).
\end{equation*}
\end{lemma}
\begin{proof}
We use the methods known from Lemma 2.1 in \cite{OK}. \par
From (\ref{e1.1}), a direct computation shows that
$$\frac{\partial(\dot{u}) }{\partial t}-F^{ij}\partial_{ij}(\dot{u})=0.$$
Using the  maximum principle, we see that
$$\min_{\bar{\Omega}_{T}}(\dot{u}) =\min_{\partial\bar{\Omega}_{T}}(\dot{u}).$$
Without loss of generality, we assume that $\dot{u}\neq constant$. If $\exists x_{0}\in \partial\Omega, t_{0}>0,$
such that $\dot{u}(x_{0},t_{0})=\min_{\bar{\Omega}_{T}}(\dot{u}).$ On one hand, since $\langle\beta, \nu\rangle>0$, by the Hopf Lemma (cf.\cite{LL}) for parabolic equations,
there must hold in the following $$\dot{u}_{\beta}(x_{0},t_{0})\neq 0.$$ On the other hand,  we differentiate the boundary condition and then obtain
$$\dot{u}_{\beta}=h_{p_{k}}(Du)\frac{\partial \dot{u}}{\partial x_{k}}=\frac{\partial h(Du)}{\partial t}=0.$$
It is  a contradiction.
So we deduce that
$$\dot{u}\geq \min_{\bar{\Omega}_{T}}(\dot{u})
=\min_{\partial\bar{\Omega}_{T}\mid_{t=0}}(\dot{u})=\min_{\bar{\Omega}}F(D^{2}u_{0}).$$
For the same reason, we have $$\frac{\partial u}{\partial t}\leq\Theta_{1}\triangleq\max_{\bar{\Omega}}F(D^{2}u_{0}).$$
Putting these facts together, the assertion follows.
\end{proof}
\begin{lemma}\label{l5.1b}
 Let $(x,t)$  be an arbitrary point of $\Omega_{T}$, and $\lambda_{1}\leq\lambda_{2}\leq\cdots\leq\lambda_{n}$
be the eigenvalues of $D^{2}u$ at $(x,t)$. As long as the convex solution to (\ref{e1.1})-(\ref{e1.3}) exists,
then $\exists \mu_{1}>0, \mu_{2}>0$ depending only on $F(D^{2}u_{0})$,  such that
$$\lambda_{1}\leq \mu_{1},  \,\,\, \lambda_{n}\geq \mu_{2}.$$
\end{lemma}
\begin{proof}
Using Lemma \ref{l5.1a} and  condition (\ref{e1.15a}),   we obtain
$$f_{1}(\lambda_{1})\leq \Theta_{1},\,\,\,f_{2}(\lambda_{n})\geq\Theta_{0}.$$
By the condition (\ref{e1.15b}), we get the desired results.
\end{proof}
By Lemma \ref{l5.1b}, the points $(\lambda_{1},\lambda_{2},\cdots, \lambda_{n})$ are always in $ \Gamma^{+}_{]\mu_{1},\mu_{2}[}$  under the flow. So we obtain the next lemma.
\begin{lemma}\label{l5.1c}
 Let $(x,t)$  be an arbitrary point of $\Omega_{T}$, and $\lambda_{1}\leq\lambda_{2}\leq\cdots\leq\lambda_{n}$
be the eigenvalues of $D^{2}u$ at $(x,t)$. As long as the convex solution to (\ref{e1.1})-(\ref{e1.3}) exists,
then $\exists \lambda>0, \Lambda>0$ depending only on $F(D^{2}u_{0})$,  such that $F$ satisfies the structure conditions (\ref{e1.16}), (\ref{e1.17}).
\end{lemma}
In the following, we always assume that  $\lambda$ and $ \Lambda$ are universal constants depending on the known data. In order to establish the $C^{2}$ estimates, first we make use of the method to do the strict obliqueness estimates, a parabolic version of
a result of J.Urbas \cite{JU} which  was given in \cite{OK}. Returning to Lemma \ref{l1.3}, we  get a uniform positive lower bound
of the quantity $\inf_{\partial\Omega}h_{p_{k}}(Du)\nu_{k}$ which does not depend on $t$ under the structure
conditions of $F$.
\begin{lemma}\label{l3.4}
As long as the strictly convex solution to (\ref{e1.1})-(\ref{e1.3}) exists,  the strict obliqueness estimates
can be obtained by
\begin{equation}\label{e3.4}
\langle\beta, \nu\rangle\geq \frac{1}{C_{1}}>0,
\end{equation}
where the constant $C_{1}$ is independent of $t$.
\end{lemma}
\begin{proof}
Let $(x_{0},t_{0})\in \partial\Omega\times[0,T]$ such that
$$\langle\beta, \nu\rangle(x_{0},t_{0})=h_{p_{k}}(Du)\nu_{k}=\min_{\partial\Omega\times[0,T]}\langle\beta, \nu\rangle.$$
By the computation in \cite{JU}, we know
\begin{equation}\label{e3.0}
\langle\beta, \nu\rangle=\sqrt{u^{ij}\nu_{i}\nu_{j}h_{p_{k}}h_{p_{l}}u_{kl}}.
\end{equation}
Further on, we may assume that $t_{0}>0$ and $\nu(x_{0})=(0,0,\cdots,1)\triangleq e_{n}$. As in
the proof of Lemma 8.1 in \cite{OK}, by the convexity of $\Omega$ and its smoothness, we extend $\nu$ smoothly to a tubular neighborhood of $\partial\Omega$ such that in matrix sense
\begin{equation}\label{e3.5}
D_{k}\nu_{l}\equiv \nu_{kl}\leq -\frac{1}{C}\delta_{kl}
\end{equation}
for some positive constant $C$. Let
$$v=\langle\beta, \nu\rangle+h(Du).$$
By the above assumptions and the boundary condition, we obtain  $$v(x_{0},t_{0})=\min_{\partial\Omega\times[0,T]}v
=\min_{\partial\Omega\times[0,T]}\langle\beta, \nu\rangle.$$
In $(x_{0},t_{0})$,  we have
\begin{equation}\label{e3.6}
0=v_{r}=h_{p_{n}p_{k}}u_{kr}+ h_{p_{k}}\nu_{kr}+h_{p_{k}}u_{kr},\quad 1\leq r\leq n-1,
\end{equation}
We assume that the following key estimate holds which will be proved later ,
\begin{equation}\label{e3.8}
v_{n}(x_{0},t_{0})\geq -C,
\end{equation}
where $C$ is a constant depending only on $\Omega$, $u_{0}$, $h$, $\epsilon_{0}$,  $\tilde{h}$ and we will use the conditions of
(\ref{e1.4aa}), (\ref{e1.15}), (\ref{e1.16}), (\ref{e1.17}). It's not hard to check that
 (\ref{e3.8}) can be rewritten as
\begin{equation}\label{e3.9}
h_{p_{n}p_{k}}u_{kn}+ h_{p_{k}}\nu_{kn}+h_{p_{k}}u_{kn}\geq -C.
\end{equation}
Multiplying (\ref{e3.9}) with $h_{p_{n}}$ and (\ref{e3.6}) with $h_{p_{r}}$ respectively,
and summing up together,  we obtain
\begin{equation}\label{e3.10}
h_{p_{k}}h_{p_{l}}u_{kl}\geq -Ch_{p_{n}}-h_{p_{k}}h_{p_{l}}\nu_{kl}-h_{p_{k}}h_{p_{n}p_{l}}u_{kl}.
\end{equation}
By the concavity of $h$, we can easily check that
$$-h_{p_{n}p_{n}}\geq 0, \quad h_{p_{k}}u_{kr}=\frac{\partial h(Du)}{\partial x_{r}}=0,\quad
h_{p_{k}}u_{kn}=\frac{\partial h(Du)}{\partial x_{n}}\geq 0.$$
Substituting those into (\ref{e3.10}) and using (\ref{e3.5}),  we have
\begin{equation*}
h_{p_{k}}h_{p_{l}}u_{kl}\geq -Ch_{p_{n}}+\frac{1}{C}|Dh|^{2}=-Ch_{p_{n}}+\frac{1}{C}.
\end{equation*}
For the last term of the above inequality, we distinguish two cases in $(x_{0},t_{0})$.

Case (i).
$$ -Ch_{p_{n}}+\frac{1}{C}\leq 0.$$
Then
$$h_{p_{k}}(Du)\nu_{k}=h_{p_{n}}\geq\frac{1}{C^{2}}.$$
It shows that there is a uniform positive lower bound
for the quantity $\min_{\partial\Omega\times[0,T]}h_{p_{k}}(Du)\nu_{k}$.

Case (ii).
$$ -Ch_{p_{n}}+\frac{1}{C}>0.$$
Then we obtain a positive lower bound of $h_{p_{k}}h_{p_{l}}u_{kl}$. Introduce the Legendre transformation of $u$:
\begin{equation*}
y_{i}=\frac{\partial u}{\partial x_{i}}
,\,\,i=1,2,\cdots,n,\,\,\,u^{*}(y_{1},\cdots,y_{n},t):=\sum_{i=1}^{n}x_{i}\frac{\partial u}{\partial x_{i}}-u(x,t).
\end{equation*}
In terms of $y_{1},\cdots,y_{n}, u^{*}(y_{1},\cdots,y_{n})$, we can easily check that
$$\frac{\partial^{2} u^{*}}{\partial y_{i}\partial y_{j}}=[\frac{\partial^{2} u}{\partial x_{i}\partial x_{j}}]^{-1}.$$
Then $u^{*}$ satisfies
\begin{equation}\label{e3.12}
\left\{ \begin{aligned}\frac{\partial u^{*}}{\partial t}-F^{*}(D^{2}u^{*})&=0,
& T>t>0,\quad x\in \tilde{\Omega}, \\
\tilde{h}(Du^{*})&=0,& \qquad T>t>0,\quad x\in\partial\tilde{\Omega},\\
 u^{*}&=u^{*}_{0}, & \qquad\quad t=0,\quad x\in \tilde{\Omega},
\end{aligned} \right.
\end{equation}
where $\tilde{h}$ is a smooth and strictly concave function on $\bar{\Omega}$:
$$\Omega=\{p\in\mathbb{R}^{n} |\tilde{h}(p)>0\},\qquad |D\tilde{h}|_{{\partial\tilde{\Omega}}}=1,$$
and $u^{*}_{0}$ is the Legendre transformation of $u_{0}$. Here $F^{*}$ is defined by (\ref{e1.15}).
Set $$F^{*}(\lambda_{1},\lambda_{2},\cdots, \lambda_{n})=-F(\mu_{1},\mu_{2},\cdots, \mu_{n}),\qquad \mu_{i}=\lambda^{-1}_{i} (i=1,2,\cdots,n).$$
A simple calculation shows that
$$\sum^{n}_{i=1}\frac{\partial F^{*}}{\partial \lambda_{i}}=\sum^{n}_{i=1}\mu^{2}_{i}\frac{\partial F}{\partial \mu_{i}}
,\qquad \sum^{n}_{i=1}\lambda^{2}_{i}\frac{\partial F^{*}}{\partial \lambda_{i}}=\sum^{n}_{i=1}\frac{\partial F}{\partial \mu_{i}}.$$
Then the structure conditions of $F$ imply that $F^{*}$ also satisfies (\ref{e1.16}) and (\ref{e1.17}).
We also define
$$\tilde{v}=\tilde{\beta}^{k}\tilde{\nu}_{k}+\tilde{h}(Du^{*})=\langle\tilde{\beta}, \tilde{\nu}\rangle+\tilde{h}(Du^{*}),$$
where$$\tilde{\beta}^{k}\triangleq\frac{\partial \tilde{h}(Du^{*})}{\partial u^{*}_{k}}=\tilde{h}_{p_{k}}(Du^{*}),$$
and  $\tilde{\nu}=(\tilde{\nu}_{1}, \tilde{\nu}_{2},\cdots,\tilde{\nu}_{n})$ is the inner unit normal
 vector of $\partial\tilde{\Omega}$.
Using the same methods, under the assumption of
\begin{equation*}\label{e3.13}
\tilde{v}_{n}(y_{0},t_{0})\geq -C,
\end{equation*}
we obtain the positive lower bounds of $\tilde{h}_{p_{k}}\tilde{h}_{p_{l}}u^{*}_{kl}$
or
$$h_{p_{k}}(Du)\nu_{k}=\tilde{h}_{p_{k}}(Du^{*})\tilde{\nu}_{k}(y_{0})=\tilde{h}_{p_{n}}\geq\frac{1}{C^{2}}.$$
We notice that
$$\tilde{h}_{p_{k}}\tilde{h}_{p_{l}}u^{*}_{kl}=\nu_{i}\nu_{j}u^{ij}.$$
Then the claim follows from (\ref{e3.0}) by the positive lower bounds of $h_{p_{k}}h_{p_{l}}u_{kl}$ and $\tilde{h}_{p_{k}}\tilde{h}_{p_{l}}u^{*}_{kl}$.

It remains to prove the key estimate (\ref{e3.8}). We generalize the proof of Lemma 8.1 in \cite{OK} for the goal.

Define the linearized operator by
$$L=F^{ij}\partial_{ij}-\partial_{t}.$$
Since $D^{2}\tilde{h}\leq-\tilde{\theta}I$,  we obtain
\begin{equation}\label{e3.151}
\aligned
L\tilde{h} &\leq  -\tilde{\theta}\sum F^{ii}.\\
\endaligned
\end{equation}
On the other hand,
\begin{equation}\label{e3.100}
\aligned
Lv=&h_{p_{k}p_{l}p_{m}}\nu_{k}F^{ij}u_{li}u_{mj}+2h_{p_{k}p_{l}}F^{ij}\nu_{kj}u_{li}
+h_{p_{k}p_{l}}F^{ij}u_{lj}u_{ki}\\
&+h_{p_{k}p_{l}}\nu_{k}Lu_{l}+h_{p_{k}}L\nu_{k}.
\endaligned
\end{equation}
By estimating the first term in the diagonal basis, we have
$$\mid h_{p_{k}p_{l}p_{m}}\nu_{k}F^{ij}u_{li}u_{mj}\mid\leq C\sum^{n}_{i=1}\frac{\partial F}{\partial \lambda_{i}}\lambda^{2}_{i}\leq C,$$
where we use the assumption of (\ref{e1.17}) and  $C$ is a constant depending only on $h$, $\Omega$, $\lambda$, $\Lambda$.
For the second term, by Cauchy inequality, we obtain
\begin{equation*}
\begin{aligned}
\mid2h_{p_{k}p_{l}}F^{ij}\nu_{kj}u_{li}\mid\leq C\sum^{n}_{i=1}\frac{\partial F}{\partial \lambda_{i}}\lambda_{i}
&=C\sum^{n}_{i=1}\sqrt{\frac{\partial F}{\partial \lambda_{i}}}\sqrt{\frac{\partial F}{\partial \lambda_{i}}}\lambda_{i}\\
&\leq C(\sum^{n}_{i=1}\frac{\partial F}{\partial \lambda_{i}})(\sum^{n}_{i=1}\frac{\partial F}{\partial \lambda_{i}}\lambda^{2}_{i})\\
&\leq C.
\end{aligned}
\end{equation*}
By the same reason, we get
$$ \mid h_{p_{k}p_{l}}F^{ij}u_{lj}u_{ki}\mid\leq C.$$
After  simple calculation,  we give
$$Lu_{l}=0.$$
Obviously we have
$$|h_{p_{k}}L\nu_{k}|\leq C\sum F^{ii}.$$
So there exists a positive constant $C$ such that
\begin{equation}\label{e3.15}
Lv\leq C\sum F^{ii}.
\end{equation}
Here we use the condition (\ref{e1.16}) and $C$ depends  only on $h$, $\Omega$, $\lambda$, $\Lambda$.

Denote a neighborhood of $x_{0}$:
$$\Omega_{\delta}\triangleq \Omega\cap B_{\delta}(x_{0}),$$
where $\delta$ is a positive constant such that $\nu$ is well defined in $\Omega_{\delta}$.
We consider
$$\Phi\triangleq v(x,t)-v(x_{0},t_{0})+C_{0}\tilde{h}(x)+A|x-x_{0}|^{2},$$
where $C_{0}$ and $A$ are positive constants to be determined. On $\partial\Omega\times[0,T)$,
it is clear that $\Phi\geq 0.$ Since $v$ is bounded, we can select $A$
large enough such that
\begin{equation*}
\begin{aligned}
&(v(x,t)-v(x_{0},t_{0})+C_{0}\tilde{h}(x)+A|x-x_{0}|^{2})|_{(\Omega\cap\partial B_{\delta}(x_{0}))\times[0,T]}\\
 &\geq v(x,t)-v(x_{0},t_{0})+A\delta^{2}\\
 &\geq 0.
\end{aligned}
\end{equation*}
Using the strict concavity of $\tilde{h}$, we have
$$ \triangle(C_{0}\tilde{h}(x)+A|x-x_{0}|^{2})\leq C(-C_{0}\tilde{\theta}+2A)\sum F^{ii}.$$
Then by choosing the constant $C_{0}\gg A$,
we can show that
\begin{equation*}
\begin{aligned}
\triangle(v(x,0)-v(x_{0},t_{0})+C_{0}\tilde{h}(x)+A|x-x_{0}|^{2})\leq 0.
\end{aligned}
\end{equation*}
We calculate by using the maximum principle to get
\begin{equation*}
\begin{aligned}
&(v(x,0)-v(x_{0},t_{0})+C_{0}\tilde{h}(x)+A|x-x_{0}|^{2})|_{\Omega_{\delta}}\\
&\geq\min_{(\partial\Omega\cap B_{\delta}(x_{0}))\cup(\Omega\cap\partial B_{\delta}(x_{0})} (v(x,0)-v(x_{0},t_{0})+C_{0}\tilde{h}(x)+A|x-x_{0}|^{2})
\\
&\geq 0.
\end{aligned}
\end{equation*}
Combining (\ref{e3.151}) with (\ref{e3.15}) and letting $C_{0}$ be large enough,  we obtain $$L\Phi\leq (-C_{0}\tilde{\theta}+C+2A)\sum F^{ii}\leq 0.$$
 From the above arguments, we verify that
$\Phi$ satisfies
\begin{equation}\label{e3.16}
\left\{ \begin{aligned}L\Phi&\leq 0,\qquad
&(x,t)\in\Omega_{\delta}\times[0,T] , \\
\Phi&\geq 0,\qquad &(x,t)\in(\partial\Omega_{\delta}\times[0,T]\cup(\Omega_{\delta}\times\{t=0\}.
\end{aligned} \right.
\end{equation}
Using the maximum principle, we deduce that $$\Phi\geq 0,\qquad (x,t)\in\Omega_{\delta}\times[0,T].$$
Combining the above inequality with $\Phi(x_{0},t_{0})=0$, we obtain $\langle\nabla\Phi,e_{n}\rangle|_{(x_{0},t_{0})}\geq0$
which gives the desired key estimate (\ref{e3.8}).  Thus we complete the proof of the lemma.
\end{proof}

By making use of  ({\ref{e3.15}),   we can state the following result which is similar to Proposition 2.6 in \cite{SM}.
\begin{lemma}\label{l3.0}
Fix a smooth function $H: \Omega\times\tilde{\Omega}\rightarrow R$ and define $\varphi(x,t)=H(x,Du(x,t))$. Then
there holds
\begin{equation*}\label{e3.170}
|L\varphi|\leq C\sum F^{ii},\quad (x,t)\in \Omega_{T},
\end{equation*}
where $C$ is a positive constant depending on $h$, $H$, $\Omega$, $\lambda$, $\Lambda$.
\end{lemma}
We  now proceed to carry out the $C^{2}$  estimates. The procedure is similar to a priori estimates on the second boundary value problem for Lagrangian mean curvature flow \cite{HR}. The strategy is to
 bound the interior second derivative first.
\begin{lemma}\label{l3.5}
For each $t\in [0,T]$,  the following estimates hold:
\begin{equation}\label{e3.170}
\sup_{\Omega}\mid D^{2}u\mid\leq \max_{\partial\Omega\times[0,T]}\mid D^{2}u\mid+\max_{\bar{\Omega}}\mid D^{2}u_{0}\mid.
\end{equation}
\end{lemma}
\begin{proof}
Given any unit vector $\xi$,  according to the concavity of F,  we know that $u_{\xi\xi}$ satisfies
$$\partial_{t}u_{\xi\xi}-F^{ij}\partial_{ij}u_{\xi\xi}
=\frac{\partial^{2}F}{\partial u_{ij}\partial u_{kl}}u_{ij\xi}u_{kl\xi}\leq 0.$$
Combining with the convexity of $u$ and using the maximum principle, we obtain
\begin{equation*}
\begin{aligned}
0\leq |u_{\xi\xi}|=u_{\xi\xi}(x,t)&\leq \max_{\partial\Omega_{T}}u_{\xi\xi}\\
 &\leq\max_{\partial\Omega\times[0,T]}\mid D^{2}u\mid+\max_{\bar{\Omega}}\mid D^{2}u_{0}\mid.
\end{aligned}
\end{equation*}
Therefore we obtain the desired estimates (\ref{e3.170}).
\end{proof}
By taking tangential differentiation on  the boundary
condition $h(Du)=0$,  we have
\begin{equation}\label{e3.17}
u_{\beta\tau}=h_{p_{k}}(Du)u_{k\tau}=0,
\end{equation}
where $\tau$  denotes a tangential  vector.  The second order derivative estimates on
the boundary are controlled by $u_{\beta\tau}, u_{\beta\beta}, u_{\tau\tau}$. In the following,
 we give the arguments as in \cite{JU} which can be found in \cite{HR}. For $x\in\partial\Omega$, any unit vector $\xi$ can be written in terms of
 a tangential component $\tau(\xi)$ and a component in the direction $\beta$ by
 $$\xi=\tau(\xi)+\frac{\langle\nu,\xi\rangle}{\langle\beta,\nu\rangle}\beta,$$
 where
 $$\tau(\xi)=\xi-\langle\nu,\xi\rangle\nu-\frac{\langle\nu,\xi\rangle}{\langle\beta,\nu\rangle}\beta^{T}$$
 and
 $$\beta^{T}=\beta-\langle\beta,\nu\rangle\nu.$$
After a simple computation, we get
\begin{equation}\label{e3.19}
\begin{aligned}
|\tau(\xi)|^{2}&=1-(1-\frac{|\beta^{T}|^{2}}{\langle\beta,\nu\rangle^{2}})\langle\nu,\xi\rangle^{2}
-2\langle\nu,\xi\rangle\frac{\langle\beta^{T},\xi\rangle}{\langle\beta,\nu\rangle}\\
&\leq 1+C\langle\nu,\xi\rangle^{2}-2\langle\nu,\xi\rangle\frac{\langle\beta^{T},\xi\rangle}{\langle\beta,\nu\rangle}\\
&\leq C,
\end{aligned}
\end{equation}
where we use the strict obliqueness estimates (\ref{e3.4}). Let $\tau\triangleq \frac{\tau(\xi)}{|\tau(\xi)|}$.
Then by (\ref{e3.17}) and (\ref{e3.4}), we obtain
\begin{equation}\label{e3.20}
\begin{aligned}
u_{\xi\xi}&=|\tau(\xi)|^{2}u_{\tau\tau}+2|\tau(\xi)|\frac{\langle\nu,\xi\rangle}{\langle\beta,\nu\rangle}u_{\beta\tau}+
\frac{\langle\nu,\xi\rangle^{2}}{\langle\beta,\nu\rangle^{2}}
u_{\beta\beta}\\
&=|\tau(\xi)|^{2}u_{\tau\tau}+\frac{\langle\nu,\xi\rangle^{2}}{\langle\beta,\nu\rangle^{2}}
u_{\beta\beta}\\
&\leq C(u_{\tau\tau}+u_{\beta\beta}).
\end{aligned}
\end{equation}
Along with specifying the boundary conditions,
we can carry out the double derivative estimates in the direction $\beta$.
\begin{lemma}\label{l3.6}
For any $t\in [0,T]$, we have the estimates
 $$\max_{\partial\Omega} u_{\beta\beta}\leq C_{2},$$
 where $C_{2}>0$ depends only on $u_{0}$, $h$, $\tilde{h}$, $\Omega$, $\lambda$, $\Lambda$.
\end{lemma}
\begin{proof}
We use the  barrier functions  for any $x_{0}\in\partial\Omega$ and thus
consider
$$\Psi\triangleq\pm h(Du)+C_{0}\tilde{h}+A|x-x_{0}|^{2}.$$
 As in  the proof of (\ref{e3.16}),
we can find  constants $C_{0}$ and  $A$,  such that
\begin{equation*}\label{e3.21}
\left\{ \begin{aligned}L\Psi&\leq 0,\qquad
&(x,t)\in\Omega_{\delta}\times[0,T] , \\
\Psi&\geq 0,\qquad &(x,t)\in(\partial\Omega_{\delta}\times[0,T]\cup(\Omega_{\delta}\times\{t=0\}.
\end{aligned} \right.
\end{equation*}
By  the maximum principle,  we get
$$\Psi\geq 0,\qquad (x,t)\in\Omega_{\delta}\times[0,T].$$
Combining the above inequality with $\Psi(x_{0},t_{0})=0$  and using Lemma \ref{l3.4}, we obtain $\Psi_{\beta}(x_{0},t_{0})\geq 0$.
Furthermore we see from $\beta=(\frac{\partial h}{\partial p_{1}}, \frac{\partial h}{\partial p_{2}}, \cdots,\frac{\partial h}{\partial p_{n}})$ that
$$\frac{\partial h}{\partial\beta}=\langle Dh(Du),\beta\rangle=\Sigma_{k,l}\frac{\partial h}{\partial p_{k}}u_{kl}\beta^{l}
=\Sigma_{k,l}\beta^{k}u_{kl}\beta^{l}=u_{\beta\beta}.$$
Then we show that
$$|u_{\beta\beta}|=|\frac{\partial h}{\partial\beta}|\leq C_{2}.$$
\end{proof}
We are now in a position to obtain  the bound of double tangential derivative on the boundary.
\begin{lemma}\label{l3.7}
There exists a constant $C_{3}>0$ depending only on $u_{0}$, $h$, $\tilde{h}$, $\Omega$,  $\lambda$, $\Lambda$,  such that
 $$\max_{\partial\Omega\times[0,T]}\max_{|\tau|=1, \langle\tau,\nu\rangle=0} u_{\tau\tau}\leq C_{3}.$$
\end{lemma}
\begin{proof}
Assume that $x_{0}\in\partial\Omega$, $t_{0}\in[0,T]$  and $\nu(x_{0})=e_{n}$ is the inner unit normal of $\partial\Omega$
at $x_{0}$. We can also choose the coordinates such that
$$\max_{\partial\Omega\times[0,T]}\max_{|\tau|=1, \langle\tau,\nu\rangle=0} u_{\tau\tau}=u_{11}(x_{0},t_{0}).$$
For any $x\in\partial\Omega$, combining (\ref{e3.19}) with (\ref{e3.20}), we have
\begin{equation*}
\begin{aligned}
u_{\xi\xi}&=|\tau(\xi)|^{2}u_{\tau\tau}+\frac{\langle\nu,\xi\rangle^{2}}{\langle\beta,\nu\rangle^{2}}u_{\beta\beta}\\
&\leq (1+C\langle\nu,\xi\rangle^{2}-2\langle\nu,\xi\rangle\frac{\langle\beta^{T},\xi\rangle}{\langle\beta,\nu\rangle})u_{\tau\tau}
+\frac{\langle\nu,\xi\rangle^{2}}{\langle\beta,\nu\rangle^{2}}
u_{\beta\beta}\\
&\leq(1+C\langle\nu,\xi\rangle^{2}-2\langle\nu,\xi\rangle\frac{\langle\beta^{T},\xi\rangle}{\langle\beta,\nu\rangle})u_{11}(x_{0},t_{0})
+\frac{\langle\nu,\xi\rangle^{2}}{\langle\beta,\nu\rangle^{2}}u_{\beta\beta}.
\end{aligned}
\end{equation*}
Without loss of generality, we assume that $u_{11}(x_{0},t_{0})\geq 1$. Then by Lemma \ref{l3.4} and Lemma \ref{l3.6},
 we get
\begin{equation*}
\frac{u_{\xi\xi}}{u_{11}(x_{0},t_{0})}+2\langle\nu,\xi\rangle\frac{\langle\beta^{T},\xi\rangle}{\langle\beta,\nu\rangle}\leq 1+C\langle\nu,\xi\rangle^{2}.
\end{equation*}
Let  $\xi=e_{1}$, then we have
\begin{equation*}
\frac{u_{11}}{u_{11}(x_{0},t_{0})}+2\langle\nu,e_{1}\rangle\frac{\langle\beta^{T},e_{1}\rangle}{\langle\beta,\nu\rangle}\leq 1+C\langle\nu,e_{1}\rangle^{2}.
\end{equation*}
We see that the function
\begin{equation*}
w\triangleq A|x-x_{0}|^{2}-\frac{u_{11}}{u_{11}(x_{0},t_{0})}-2\langle\nu,e_{1}\rangle\frac{\langle\beta^{T},e_{1}\rangle}{\langle\beta,\nu\rangle}
+C\langle\nu,e_{1}\rangle^{2}+1
\end{equation*}
satisfies
$$w|_{\partial\Omega\times[0,T]}\geq 0, \quad w(x_{0},t_{0})=0.$$
As before, by (\ref{e3.170}),  we can select the constant $A$  such that
$$w|_{(\partial B_{\delta}(x_{0})\cap\Omega)\times[0,T]}\geq 0.$$
Consider
$$-2\langle\nu,e_{1}\rangle\frac{\langle\beta^{T},e_{1}\rangle}{\langle\beta,\nu\rangle}+C\langle\nu,e_{1}\rangle^{2}+1$$
as a known function depending on $x$ and $Du$. Then by Lemma \ref{l3.0}, we obtain
$$|L(-2\langle\nu,e_{1}\rangle\frac{\langle\beta^{T},e_{1}\rangle}{\langle\beta,\nu\rangle}
+C\langle\nu,e_{1}\rangle^{2}+1)|\leq C\sum F^{ii}.$$
Combining the above inequality with the proof of Lemma \ref{l3.5},  we have
$$Lw\leq C\sum F^{ii}.$$
As in the proof of Lemma \ref{l3.6}, we consider the function
$$\Upsilon\triangleq  w+C_{0}\tilde{h}.$$
Using the standard barrier argument, we show that $$\Upsilon_{\beta}(x_{0},t_{0})\geq 0.$$
 A direct computation deduce that
\begin{equation}\label{e3.22}
u_{11\beta}\leq Cu_{11}(x_{0},t_{0}).
\end{equation}
On the other hand, differentiating $h(Du)$ twice in the direction $e_{1}$ at $(x_{0},t_{0}),$
 we have

$$h_{p_{k}}u_{k11}+h_{p_{k}p_{l}}u_{k1}u_{l1}=0.$$
The concavity of $h$ yields
$$h_{p_{k}}u_{k11}=-h_{p_{k}p_{l}}u_{k1}u_{l1}\geq \tilde{C}u_{11}(x_{0},t_{0})^{2}.$$
Combining the above inequality with $h_{p_{k}}u_{k11}=u_{11\beta}$, and using (\ref{e3.22}), we obtain
\begin{equation*}\label{e3.23}
\tilde{C}u_{11}(x_{0},t_{0})^{2}\leq C u_{11}(x_{0},t_{0}).
\end{equation*}
Then we get the upper bound estiamtes of $u_{11}(x_{0},t_{0})$ and the conclusion follows.
\end{proof}
Using Lemma {\ref{l3.6}, {\ref{l3.7}, and (\ref{e3.20}), we obtain the $C^{2}$ a priori bound for the solution on the  boundary:
\begin{lemma}\label{l3.8}
There exists a constant $C_{4}>0$ depending on $h$, $\tilde{h}$, $u_{0}$,  $\lambda$, $\Lambda$, and $\Omega$, such that
\begin{equation*}\label{e3.24}
\sup_{\partial\Omega_{T}}|D^{2}u|\leq C_{4}.
\end{equation*}
\end{lemma}

Combining Lemma \ref{l3.5} with Lemma \ref{l3.8}, we obtain the following result:
\begin{lemma}\label{l3.9}
There exists a constant $C_{5}>0$ depending on $h$, $\tilde{h}$, $u_{0}$, $\Omega$,  $\lambda$, $\Lambda$ such that
\begin{equation*}\label{e3.25}
\sup_{\bar{\Omega}_{T},|\xi|=1}D_{ij}u\xi_{i}\xi_{j}\leq C_{5}.
\end{equation*}
\end{lemma}
By the Legendre transformation of $u$, using (\ref{e3.12}) and repeating the proof of the above lemmas, we get
the  following statement:
\begin{lemma}\label{l3.10}
There exists a constant $C_{6}>0$ depending on $h$, $\Omega$, $\tilde{h}$, $\tilde{\Omega}$, $u_{0}$, $\lambda$, $\Lambda$ such that
\begin{equation}\label{e3.26}
\frac{1}{C_{6}}\leq\inf_{\bar{\Omega}_{T},|\xi|=1}D_{ij}u\xi_{i}\xi_{j}
\leq\sup_{\bar{\Omega}_{T},|\xi|=1}D_{ij}u\xi_{i}\xi_{j}\leq C_{6}.
\end{equation}
\end{lemma}
\begin{rem}\label{r3.9}
The differential inequality (\ref{e3.15}) plays a central role
in the $C^{2}$ estimates for the solution on the boundary. The conditions (\ref{e1.16}) and (\ref{e1.17}) imposed on $F$  is used to
prove (\ref{e3.15}). If we assume (\ref{e3.15}), (\ref{e1.4aa}) and (\ref{e1.15}), we can also prove the above lemma instead of
(\ref{e1.16}) and (\ref{e1.17}).
\end{rem}

\section{Longtime existence and convergence}
{\bf Proof of Theorem \ref{t1.1}:}

We divide the proof of the Theorem into two parts.\\

Part 1: The long time existence. \\

By Lemma \ref{l5.1a} and Lemma \ref{l3.9},
we know global $C^{2,1}$ estimates for the solutions of the flow \eqref{e2.1}.
Using Theorem 14.22 in Lieberman \cite{GM}, we show that the solutions of the oblique derivative problem \eqref{e2.1}
have global  $C^{2+\alpha,1+\frac{\alpha}{2}}$ estimates.

Now let $u_{0}$  be  a  $C^{2+\alpha_{0}}$ strictly convex function as in the conditions of Theorem \ref{t1.1}.
We assume that $T$ is the maximal time such that the solution to the  flow (\ref{e2.1}) exists. Suppose that $T<+\infty$.
Combining Proposition \ref{p1.1} with Lemma \ref{l3.10} and using Theorem 14.23  in \cite{GM} ,
there exists $ u\in C^{2+\alpha,1+\frac{\alpha}{2}}(\bar{\Omega}_{T})$  which satisfies (\ref{e2.1}) and
 $$\|u\|_{C^{2+\alpha,1+\frac{\alpha}{2}}(\bar{\Omega}_{T})}<+\infty.$$
 Then we can extend the flow (\ref{e2.1}) beyond the maximal time $T$.
 So that we deduce that $T=+\infty.$ Then there  exists the solution $u(x,t)$  for all times $t>0$ to (\ref{e1.1})-(\ref{e1.3}).\\

 Part 2: The convergence. \\

Using the boundary condition, we have
\begin{equation}\label{e4.2099}
|Du|\leq C_{7} ,
\end{equation}
where $C_{7}$ be a constant depending on $\Omega$ and $\tilde{\Omega}$. By intermediate Schauder estimates for parabolic
equations (cf. Lemma 14.6  and Proposition 4.25 in \cite{GM}), for any $D\subset\subset \Omega$, we have
\begin{equation*}\label{e4.209}
[D^{2}u]_{\alpha,\frac{\alpha}{2}, D_{T}}\leq C\sup|D^{2}u|\leq C_{8},
\end{equation*}
\begin{equation}\label{e4.209}
\begin{aligned}
&\sup_{t\geq 1}\|D^{3}u(\cdot,t)\|_{C(\bar{D})}+\sup_{t\geq 1}\|D^{4}u(\cdot,t)\|_{C(\bar{D})}\\
&+\sup_{x_{i}\in D, t_{i}\geq 1}\frac{|D^{4}u(x_{1}, t_{1})-D^{4}u(x_{2}, t_{2})}{\max\{|x_{1}-x_{2}|^{\alpha},|t_{1}-t_{2}|^{\frac{\alpha}{2}}\}}\\
&\leq C_{9},
\end{aligned}
\end{equation}
where $ C_{8}$, $ C_{9}$ are  constants depending on the known data and dist$(\partial \Omega, \partial D)$.

To finish the proof of Theorem \ref{t1.1}, we use a trick that we learned from  J. Kitagawa\cite{JK} and O.C. Schn$\ddot{\text{u}}$rer\cite{OC} which proving that
some curvature flows with second boundary condition converge to translating solutions. Now fixing some positive $t_0$ and writing $$v(x,t)=u(x,t)-u(x,t+t_0),$$ 
then we have
\begin{equation}\label{e4.20}
\begin{aligned}
\dot{v}&=F(D^2u(x,t)-F(D^2u(x,t+t_0))\\
&=\int_0^1\frac{d}{ds}F(sD^2u(x,t)+(1-s)D^2u(x,t+t_0))ds\\
&=\big(\int_0^1 \nabla_{r_{ij}}F(sD^2u(x,t)+(1-s)D^2u(x,t+t_0))ds\big)v_{ij}\\
&=a^{ij}v_{ij},
\end{aligned}
\end{equation}
where we denote
\begin{equation*}
a^{ij}=\int_0^1 \nabla_{r_{ij}}F(sD^2u(x,t)+(1-s)D^2u(x,t+t_0))ds.
\end{equation*}
Since the smallest eigenvalue of $u_{ij}$ has a strictly positive lower bound and upper bound uninformly in $t$ and $x$ by \eqref{e3.26}
and $u\in C^{2+\alpha,1+\frac{\alpha}{2}}(\bar{\Omega}_{T})$,  we see that
\begin{equation*}
||u(\cdot,t)-u(\cdot,t+t_0)||_{C^2(\bar{\Omega})}\leq C\cdot t_0^{\frac{\alpha}{2}},
\end{equation*}
where $C$ is independent of $t$ by  Theorem 14.22 in Lieberman \cite{GM}.

By taking $t_0$ sufficiently small, we can ensure that the convex combination $sD^2u(\cdot,t)+(1-s)D^2u(\cdot,t+t_0)$
is close to $u(\cdot,t)$ in $C^2(\bar{\Omega})$ norm. By the uniform positive lower bound and upper bound on the eigenvalues of $u_{ij}$
in Lemma \ref{l3.10},
 there exist positive constants $\tilde{\lambda}$, $\tilde{\Lambda}$ such that
$$\tilde{\lambda}I\leq[a^{ij}]\leq \tilde{\Lambda}I$$
 and hence the equation (\ref{e4.20}) is uniformly parabolic.

Additionally, we see that for $x\in\partial \Omega$, $v$ satisfies
\begin{equation}\label{e4.200}
\begin{aligned}
0&=h(Du(x,t))-h(Du(x, t+t_0))\\
&=(\int_0^1h_{p_k}(x,sDu(x,t)+(1-s)Du(x,t+t_0))ds)v_k\\
&=:\alpha^kv_k.
\end{aligned}
\end{equation}
By Lemma 3.4, we  have  $h_{p_k}(x,Du(x,t))\nu_k\geq C>0$ for some $C$ uniform in $t$ and $x$.
Therefore as in the above, by choosing $t_0$ small enough, we can deduce that $\alpha^k\nu_k\geq \frac{C}{2}>0$
and we see that $v$ satisfies a linear, uniformly oblique boundary condition.

By the proof of section 6.2 in \cite{OC},
the uniformly parabolic equation (\ref{e4.20}) with the uniformly oblique boundary condition (\ref{e4.200}) ensures that
we can obtain a translating solution of the same regularity as $u$ : $u^{\infty}(x,t)=u^\infty(x)+C_{\infty}\cdot t$
for some constant $C_{\infty}$ which satisfies equation \eqref{e1.1} and $u^{\infty}(x,t)$ satisfies
\begin{equation}\label{e4.4}
\lim_{t\rightarrow+\infty}\|u(\cdot,t)-u^{\infty}(\cdot,t)\|_{C(\bar{\Omega})}=0.
\end{equation}
Using (\ref{e4.4}) and the interpolation inequalities of the following form: (cf. \cite{OC})
\begin{equation}\label{e4.5}
\begin{aligned}
&\|D(u(\cdot,t)-u^{\infty}(\cdot,t))\|^{2}_{C(\bar{\Omega})}\\
&\leq c(\Omega)\|u(\cdot,t)-u^{\infty}(\cdot,t)\|_{C(\bar{\Omega})}(\|D^{2}(u(\cdot,t)-u^{\infty}(\cdot,t))\|_{C(\bar{\Omega})}\\
&+\|D(u(\cdot,t)-u^{\infty}(\cdot,t))\|_{C(\bar{\Omega})})\\
&\leq c(\Omega)\|u(\cdot,t)-u^{\infty}(\cdot,t)\|_{C(\bar{\Omega})}(2n^{2}C_{6}+2C_{7}),
\end{aligned}
\end{equation}
 we obtain
\begin{equation}\label{e4.6}
\lim_{t\rightarrow+\infty}\|u(\cdot,t)-u^{\infty}(\cdot,t)\|_{C^{1}(\bar{\Omega})}=0.
\end{equation}
By the interpolation inequalities  and (\ref{e4.209}), we get
\begin{equation*}\label{e4.7}
\begin{aligned}
&\|D^{2}(u(\cdot,t)-u^{\infty}(\cdot,t))\|^{2}_{C(\bar{D})}\\
&\leq c(D)\|D(u(\cdot,t)-u^{\infty}(\cdot,t))\|_{C(\bar{D})}(\|D^{3}(u(\cdot,t)-u^{\infty}(\cdot,t))\|_{C(\bar{D})}\\
&+\|D^{2}(u(\cdot,t)-u^{\infty}(\cdot,t))\|_{C(\bar{D})})\\
&\leq c(D)\|D(u(\cdot,t)-u^{\infty}(\cdot,t))\|_{C(\bar{D})}(2C_{9}+2n^{2}C_{6}).
\end{aligned}
\end{equation*}
Using (\ref{e4.6}),  we also obtain
\begin{equation}\label{e4.8}
\lim_{t\rightarrow+\infty}\|u(\cdot,t)-u^{\infty}(\cdot,t)\|_{C^{2}(\bar{D})}=0.
\end{equation}
Repeating the above procedure and using (\ref{e4.209}),  we have
\begin{equation}\label{e4.9}
\lim_{t\rightarrow+\infty}\|u(\cdot,t)-u^{\infty}(\cdot,t)\|_{C^{4+\alpha}(\bar{D})}=0.
\end{equation}
By Lemma \ref{l3.10} and (\ref{e4.8}), we get
$$\lim_{t\rightarrow+\infty}\|u(\cdot,t)-u^{\infty}(\cdot,t)\|_{C^{1+\zeta}(\bar{\Omega})}=0.$$
Therefore by using the equation \eqref{e1.1} and letting $t\rightarrow\infty$, we obtain
 \begin{equation*}
C_{\infty} =\frac{\partial u^{\infty}(x,t)}{\partial t}= F(D^2u^{\infty}(x,t))=F(D^2u^{\infty}(x)),
 \end{equation*}
 $$0=\lim_{t\rightarrow\infty}h(Du(x,t))=\lim_{t\rightarrow\infty}h(Du^{\infty}(x,t))=h(Du^{\infty}(x)).$$
Then the proof of Theorem \ref{t1.1} is completed.
\qed
\\

{\bf Proof of Corollary \ref{r1.2}:}
As the arguments in Section 2, we know that (\ref{e1.18}) is equivalent to
\begin{equation}\label{e4.1}
\left\{ \begin{aligned}F(D^{2}u)&=C_{\infty},\,x\in \Omega, \\
h(Du)&=0,\,\,x\in\partial\Omega.
\end{aligned} \right.
\end{equation}
By Evans-Krylov theorem and interior Schauder estimates,   we have
$$u\in  C^{\infty}(\Omega).$$
 For any $l\in\{1,2,\cdots,n\},$ set $w=u_{l}$.  Then $w$ satisfies
\begin{equation*}\label{e4.2}
\left\{ \begin{aligned}F^{ij}w_{ij}&=0,\,\,x\in\Omega,\\
\beta^{k}w_{k}&=0,\,\,x\in\partial\Omega,
\end{aligned} \right.
\end{equation*}
where
$$\langle\beta, \nu\rangle\geq \frac{1}{C_{1}}>0$$
by Lemma \ref{l3.4} and $\nu=(\nu_{1},\nu_{2}, \cdots,\nu_{n})$ is the unit inward normal vector of $\partial\Omega$.
Using Theorem 6.30 in Gilbarg-Trudinger \cite{GT}, we have $w\in C^{2+\alpha}(\bar{\Omega})$. Then $$u\in  C^{3+\alpha}(\bar{\Omega}).$$
By the smoothness of $F$ and $\partial\Omega$ and  using bootstrap argument,  we obtain
$$u\in  C^{\infty}(\bar{\Omega}).$$
\qed

\section{proof of theorem \ref{t1.2} and corollary \ref{c1.4}}

Let us consider the case $\tau=\frac{\pi}{4}$. Then we write
\begin{equation}\label{e5.1}
F(\lambda_{1},\lambda_{2},\cdots, \lambda_{n}):=-F_{\frac{\pi}{4}}(\lambda_{1},\lambda_{2},\cdots, \lambda_{n})=-\sum_{i}\frac{1}{1+\lambda_{i}}.
\end{equation}
The proof of Theorem \ref{t1.2} is based on the following conclusions.
\begin{lemma}\label{l5.0}\quad
\\
$F$ and $F^{*}$ are concave on the cone $\Gamma_{+}$.
\end{lemma}
\begin{proof}
We calculate directly to obtain:
$$\sum\frac{\partial^{2} F}{\partial\lambda_{i}\lambda_{j}}\xi_{i}\xi_{j}=-2\sum\frac{\xi^{2}_{i}}{(1+\lambda_{i})^{3}}\leq 0,$$
$$\sum\frac{\partial^{2} F^{*}}{\partial\lambda_{i}\lambda_{j}}\xi_{i}\xi_{j}=-2\sum\frac{\xi^{2}_{i}}{(1+\lambda_{i})^{3}}\leq 0.$$
Thus, the above inequalities  imply the desired result.
\end{proof}
By Lemma \ref{l5.1a}, we obtain the following statements.
\begin{lemma}[$\dot{u}$-estimates]\label{l5.1}\quad
\\
As long as the convex solution to (\ref{e1.1})-(\ref{e1.3}) exists, the following estimates hold, i.e.,
\begin{equation*}\label{e3.1}
\Theta_{0}\triangleq\min_{\bar{\Omega}}F_{\frac{\pi}{4}}(D^{2}u_{0})\leq-\dot{u}\triangleq-\frac{\partial u}{\partial t}\leq\Theta_{1}\triangleq\max_{\bar{\Omega}}F_{\frac{\pi}{4}}(D^{2}u_{0}).
\end{equation*}
\end{lemma}

Since $u_{0}\in C^{2+\alpha}(\bar{\Omega})$  is strictly convex, it is clear that $0<\Theta_{0}\leq \Theta_{1}<n$.
\begin{lemma}\label{l5.2} Let $(x,t)$  be an arbitrary point of $\Omega_{T}$, and $\lambda_{1}\leq\lambda_{2}\leq\cdots\leq\lambda_{n}$
be the eigenvalues of $D^{2}u$ at $(x,t)$. As long as the convex solution to (\ref{e1.1})-(\ref{e1.3}) exists, then
\begin{equation}\label{e5.2}
0\leq\lambda_{1}\leq \frac{n}{\Theta_{0}}-1, \quad \frac{n}{\Theta_{1}}-1\leq\lambda_{n}.
\end{equation}
\end{lemma}
\begin{proof}
Following the definition of $F_{\frac{\pi}{4}}(D^{2}u)$ and Lemma \ref{l5.1}:
$$\Theta_{0}\leq-\dot{u}=\sum_{i}\frac{1}{1+\lambda_{i}}\leq\frac{n}{1+\lambda_{1}}.$$
Combining with the convexity of $u$, we obtain
$$1\leq 1+\lambda_{1}\leq \frac{n}{\Theta_{0}}.$$
Using the same methods, we get
$$\frac{n}{1+\lambda_{n}}\leq-\dot{u}=\sum_{i}\frac{1}{1+\lambda_{i}}\leq\Theta_{1}.$$
Then we have
$$\frac{n}{\Theta_{1}}\leq 1+\lambda_{n}.$$
From the above arguments, we prove the lemma.
\end{proof}
Now we show the operator $-F_{\frac{\pi}{4}}$ satisfying (\ref{e1.16}) and (\ref{e1.17}) which  play an important role
in the barrier arguments.
\begin{Corollary}\label{c5.3}
As long as the convex solution to (\ref{e1.1})-(\ref{e1.3}) exists, then  for any $(x,t)\in\Omega_{T}$ we have
\begin{equation}\label{e5.3}
\frac{\Theta^{2}_{0}}{n^{2}}\leq\sum_{i=1}^{n}\frac{\partial F}{\partial\lambda_{i}}\leq n,
\end{equation}
\begin{equation}\label{e5.4}
\frac{(n-\Theta_{1})^{2}}{n^{2}}\leq\sum_{i=1}^{n}\frac{\partial F}{\partial\lambda_{i}}\lambda^{2}_{i}\leq n.
\end{equation}
\end{Corollary}
\begin{proof}
We observe
$$\sum_{i=1}^{n}\frac{\partial F}{\partial\lambda_{i}}=\sum_{i=1}^{n}\frac{1}{(1+\lambda_{i})^{2}}.$$
By Lemma \ref{l5.2} and the convexity of $u$, we have
$$\frac{\Theta^{2}_{0}}{n^{2}}\leq\frac{1}{(1+\lambda_{1})^{2}}\leq
\sum_{i=1}^{n}\frac{1}{(1+\lambda_{i})^{2}}\leq n.$$
We further obtain
$$\frac{(n-\Theta_{1})^{2}}{n^{2}}\leq\frac{\lambda^{2}_{n}}{(1+\lambda_{n})^{2}}\leq\sum_{i=1}^{n}\frac{\lambda^{2}_{i}}{(1+\lambda_{i})^{2}}\leq n.$$
Corollary \ref{c5.3} is established.
\end{proof}
{\bf Proof of Theorem \ref{t1.2}:}

It follows from Lemma \ref{l5.0},  Corollary \ref{c5.3} and Theorem \ref{t1.1} .\qed

{\bf Proof of Corollary \ref{c1.4}:}

Let $u$ be a convex solution of (\ref{e1.13}) and we define $f=D u$. Then $f$ is a diffeomorphism from $\Omega$ to $\tilde{\Omega}$.
It follows from \cite{MW}  that  $$\Sigma=\{(x,f(x))|x\in \Omega\}$$
is a special Lagrangian graph in ($\mathbb{R}^{n}\times\mathbb{R}^{n}, g_{\frac{\pi}{4}}$). The required properties are thus proved.
\qed
\section{Applications to special lagrangian   diffeomorphism problems }
Let
\begin{equation*}\label{e6.1}
F(\lambda_{1},\lambda_{2},\cdots, \lambda_{n}):=F_{\frac{\pi}{2}}(\lambda_{1},\lambda_{2},\cdots, \lambda_{n})=\sum_{i}\arctan\lambda_{i}.
\end{equation*}
A direct computation shows that $F$ and $F^{*}$ are concave on the cone  $\Gamma_{+}$. Using Corollary 3.3 in \cite{HR} we see that $F$ satisfies (\ref{e1.16}) and (\ref{e1.17}).
By Theorem \ref{t1.1}, we reprove the existence result on the second boundary value problem for special Lagrangian equations which belongs to S.Brendle and M.Warren.
\begin{theorem}(see Theorem 1.1 in \cite{SM})\label{t6.1}
Let $\Omega$, $\tilde{\Omega}$ be bounded, uniformly convex domains with smooth boundary in $\mathbb{R}^{n}$.
  Then  there exist  $u\in C^{\infty}(\bar{\Omega})$  and the constant $c$ which satisfy
\begin{equation}\label{e6.1}
\left\{ \begin{aligned}\sum_{i}\arctan\lambda_{i}&=c,
&  x\in \Omega, \\
Du(\Omega)&=\tilde{\Omega},
\end{aligned} \right.
\end{equation}
 where  the solution is unique up to additions of constants.
\end{theorem}
By \cite{SM} and \cite{MW}, we present the similar special Lagrangian  diffeomorphism problem in the following.
The two convex domains $\Omega$, $\tilde{\Omega}\subset \mathbb{R}^{n}$ with smooth boundary are fixed. Given a diffeomorphism $f$ : $\Omega \longrightarrow \tilde{\Omega}$,
 the graph $\Sigma=\{(x,f(x)): x\in \Omega\}\subset \mathbb{R}^{2n}_{n}$ is considered.

 Question: How to find a diffeomorphism $f$ : $\Omega \longrightarrow \tilde{\Omega}$ such that
 $\Sigma$ is a special Lagrangian graph in $\mathbb{R}^{2n}_{n}$. By \cite{Xin}, the graph $\sum$ is special Lagrangian  if and only if
 there is a convex potential function $u: \Omega\rightarrow \mathbb{R}$ such that $f=D u$, and
 $$\ln\det D^{2}u=c,$$
 where $c$ is some constant.
Hence, we present the problem as
\begin{equation}\label{e6.2}
\left\{ \begin{aligned}\ln\det D^{2}u&=c,
&  x\in \Omega, \\
Du(\Omega)&=\tilde{\Omega}.
\end{aligned} \right.
\end{equation}
In 1997, J.Urbas proved the existence and uniqueness  of global smooth convex solutions to (\ref{e6.2}), i.e.,
\begin{theorem}(see Theorem 1.1 in \cite{JU})\label{t6.2}
Let $\Omega$, $\tilde{\Omega}$ be bounded, uniformly convex domains with smooth boundary in $\mathbb{R}^{n}$.
  Then  there exist  $u\in  C^{\infty}(\bar{\Omega})$  and the constant $c$ which satisfy (\ref{e6.2})
  where the solution is unique up to additions of constants.
\end{theorem}
Later, O.C. Schn$\ddot{\text{u}}$rer and K. Smoczyk also obtained the above result by parabolic methods \cite{OK}.
In fact, we can consider  Lagrangian mean curvature flow in  pseudo-Euclidean space $\mathbb{R}^{2n}_{n}$ (cf.\cite{HB}) with second  boundary condition:
\begin{equation}\label{e6.3}
\left\{ \begin{aligned}\frac{\partial u}{\partial t}&=\ln\det D^{2}u,
& t>0,\quad x\in \Omega, \\
Du(\Omega)&=\tilde{\Omega}, &t>0,\qquad\qquad\\
 u&=u_{0}, & t=0,\quad x\in \Omega.
\end{aligned} \right.
\end{equation}
If we define
\begin{equation*}\label{e6.1}
F(\lambda_{1},\lambda_{2},\cdots, \lambda_{n}):=F_{\tau}(\lambda_{1},\lambda_{2},\cdots, \lambda_{n})|_{\tau=0}=\sum_{i}\ln\lambda_{i},
\end{equation*}
then a direct computation shows that  $F$ and $F^{*}$ are concave on the cone  $\Gamma_{+}$.
Here we need modify $v$ slightly which occur in Lemma \ref{e3.4}, i.e.,
$$v=\langle\beta, \nu\rangle+Ah(Du)$$
which was introduced by J.Urbas \cite{JU}.
Using the computation of Lemma 8.1 in \cite{OK}, we  have
\begin{equation*}
Lv\leq C\sum F^{ii},
\end{equation*}
where $L\triangleq F^{ij}\partial_{ij}-\partial_{t}$ as before.
By Remark \ref{r3.9}, we also obtain Theorem \ref{t6.2} as a corollary of Theorem \ref{t1.1}.

In summary, by Theorem \ref{t1.1},  we obtain the existence and uniqueness  of global smooth convex solutions to (\ref{e1.12}), i.e.,
\begin{equation*}
\left\{ \begin{aligned}F_{\tau}(D^{2}u)&=c,
&  x\in \Omega, \\
Du(\Omega)&=\tilde{\Omega},
\end{aligned} \right.
\end{equation*}
for $\tau=0, \frac{\pi}{4}$, $\frac{\pi}{2}$. Finally, we give another version of the above results as following
\begin{theorem}
Let $\Omega$, $\tilde{\Omega}$ be bounded, uniformly convex domains with smooth boundary in $\mathbb{R}^{n}$. Then
there exist  smooth diffeomorphisms $f_{\tau}$: $\Omega\rightarrow\tilde{\Omega}$ such that
 $$\Sigma=\{(x,f_{\tau}(x))|x\in \Omega\}$$
are special Lagrangian graphs in ($\mathbb{R}^{n}\times\mathbb{R}^{n}, g_{\tau}$) for $\tau=0, \frac{\pi}{4}$, $\frac{\pi}{2}$.
\end{theorem}

\vspace{5mm}

{\bf Acknowledgments:} The authors would like to thank  referees for useful
comments, which improve the paper.

\end{document}